\documentclass[a4paper,oneside,11pt]{article}
\usepackage{import}
\usepackage{fullpage}
\usepackage{fancyhdr}
\AtBeginEnvironment{pmatrix}{\setlength{\arraycolsep}{8pt}}
\usepackage{setspace}

\usepackage{amsmath,etoolbox}
\usepackage{amssymb}
\usepackage[utf8x]{inputenc} 
\usepackage{accents}
\usepackage{bbm}
\usepackage{mathtools}
\usepackage{amsthm}
\usepackage[english]{babel}
\usepackage{letltxmacro}
\usepackage{xparse}
\usepackage[dvipsnames,svgnames,x11names,hyperref]{xcolor}
\usepackage{graphicx}
\usepackage{multirow}
\usepackage{multicol}

\usepackage{tikz}
\usetikzlibrary{decorations.pathreplacing}
\usetikzlibrary{decorations.markings}
\usepackage{wrapfig}
\usepackage{caption}
\usepackage{subcaption}
\usepackage[percent]{overpic}
\usepackage{tikz-cd}

\usepackage{ebgaramond}
\usepackage{euscript}
\usepackage{mathrsfs}
\usepackage{bm}
\usepackage{commath}
\usepackage{accents}
\usepackage[bbgreekl]{mathbbol}
\usepackage[euler]{textgreek}
\usepackage[scr=boondoxo]{mathalfa}
\DeclareSymbolFontAlphabet{\mathbbm}{bbold}
\DeclareSymbolFontAlphabet{\mathbb}{AMSb}%
\usepackage{mleftright}
\usepackage{dirtytalk}
\usepackage{shuffle}
\usepackage{scalerel}
\usepackage{tensor}
\usepackage{soul}
\setul{0pt}{0.4pt}
\usepackage[T1]{fontenc}

\usepackage[colorlinks=true]{hyperref}
\definecolor{imperialBlue}{RGB}{0, 62, 116}
\definecolor{imperialBrick}{RGB}{165,25,0}
\definecolor{imperialProcess}{RGB}{0,133,202}
\definecolor{imperialGreen}{RGB}{2,137,59}
\definecolor{imperialRed}{RGB}{221,37,1}
\definecolor{imperialOrange}{RGB}{210,64,0}
\definecolor{imperialBlue2}{RGB}{0,110,175}
\definecolor{imperialTangerine}{RGB}{236,115,0}
\definecolor{imperialPurple}{RGB}{101,48,152}
\definecolor{imperialLime}{RGB}{196,214,0}
\definecolor{imperialKermit}{RGB}{102,164,10}
\hypersetup{
    colorlinks,
    citecolor=imperialGreen,
    filecolor=imperialBlue,
    linkcolor=imperialBlue,
    urlcolor=imperialProcess
}

\usepackage[nameinlink, capitalise]{cleveref}
\usepackage{mdframed}
\newtheorem{theorem}{Theorem}[section]
\newtheorem{corollary}[theorem]{Corollary}
\newtheorem{lemma}[theorem]{Lemma}
\newtheorem{proposition}[theorem]{Proposition}
\newtheorem{remark}[theorem]{Remark}
\theoremstyle{definition}
\newtheorem{definition}[theorem]{Definition}

\newtheorem{notation}[theorem]{Notation}

\newmdtheoremenv[
hidealllines=true,
leftline=true,
innertopmargin=0pt,
innerbottommargin=0pt,
linewidth=4pt,
linecolor=gray!40,
innerrightmargin=0pt,
innertopmargin=-6pt,
]{example}{Example}[section]

\usepackage{algorithm}
\usepackage[noend]{algpseudocode}

\usepackage{listings}
\usepackage{enumitem}
\usepackage[nottoc,notlot,notlof]{tocbibind}

\newcommand{\C}{\mathbb{C}}

\newcommand{\Ex}{\mathbb{E}}

\newcommand{\N}{\mathbb{N}}

\renewcommand{\P}{\mathbb{P}}

\newcommand{\R}{\mathbb{R}}
\renewcommand{\S}{\mathbb{S}}

\renewcommand{\AA}{\mathcal{A}}

\newcommand{\CC}{\mathcal{C}}
\newcommand{\DD}{\mathcal{D}}

\newcommand{\FF}{\mathcal{F}}
\newcommand{\GG}{\mathcal{G}}

\newcommand{\MM}{\mathcal{M}}

\newcommand{\OO}{\mathcal{O}}
\newcommand{\PP}{\mathcal{P}}

\renewcommand{\SS}{\mathcal{S}}
\newcommand{\TT}{\mathcal{T}}
\newcommand{\UU}{\mathcal{U}}

\newcommand{\WW}{\mathcal{W}}
\newcommand{\XX}{\mathcal{X}}

\DeclarePairedDelimiterX{\normop}[1]{\lVert}{\rVert_{\mathrm{op}}}{#1}
\DeclarePairedDelimiterX{\normopp}[1]{\lVert}{\rVert_{\mathrm{op}}^p}{#1}
\DeclarePairedDelimiterX{\normhs}[1]{\lVert}{\rVert_{\mathrm{HS}}}{#1}

\newcommand{\tr}[1]{\mathrm{tr} \Big(#1\Big)}
\newcommand{\NC}[1]{\mathrm{NC}_2\big(#1\big)}
\newcommand{\wt}[1]{\mathrm{wt} (#1)}

\newcommand{\id}{\mathbbm{1}}

\makeatletter
\def\upintkern@{\mkern-7mu\mathchoice{\mkern-3.5mu}{}{}{}}
\def\upintdots@{\mathchoice{\mkern-4mu\@cdots\mkern-4mu}%
	{{\cdotp}\mkern1.5mu{\cdotp}\mkern1.5mu{\cdotp}}%
	{{\cdotp}\mkern1mu{\cdotp}\mkern1mu{\cdotp}}%
	{{\cdotp}\mkern1mu{\cdotp}\mkern1mu{\cdotp}}}

\newcommand{\UpMultiIntegral}[1]{%
	\edef\ints@c{\noexpand\upintop
		\ifnum#1=\z@\noexpand\upintdots@\else\noexpand\upintkern@\fi
		\ifnum#1>\tw@\noexpand\upintop\noexpand\upintkern@\fi
		\ifnum#1>\thr@@\noexpand\upintop\noexpand\upintkern@\fi
		\noexpand\upintop
		\noexpand\ilimits@
	}%
	\futurelet\@let@token\ints@a
}
\makeatother

\DeclareFontFamily{OMX}{mdbch}{}
\DeclareFontShape{OMX}{mdbch}{m}{n}{ <->s * [0.8]  mdbchr7v }{}
\DeclareFontShape{OMX}{mdbch}{b}{n}{ <->s * [0.8]  mdbchb7v }{}
\DeclareFontShape{OMX}{mdbch}{bx}{n}{<->ssub * mdbch/b/n}{}

\DeclareSymbolFont{uplargesymbols}{OMX}{mdbch}{m}{n}
\SetSymbolFont{uplargesymbols}{bold}{OMX}{mdbch}{b}{n}
\DeclareMathSymbol{\upintop}{\mathop}{uplargesymbols}{82}
\DeclareMathSymbol{\upointop}{\mathop}{uplargesymbols}{"48}

\DeclareFontEncoding{MDB}{}{}
\DeclareFontFamily{MDB}{mdbch}{}
\DeclareFontShape{MDB}{mdbch}{m}{n}{ <->s * [0.8]  mdbchrmb }{}
\DeclareFontShape{MDB}{mdbch}{b}{n}{ <->s * [0.8]  mdbchbmb }{}
\DeclareFontShape{MDB}{mdbch}{bx}{n}{<->ssub * mdbch/b/n}{}
\DeclareFontSubstitution{MDB}{cmr}{m}{n}
\DeclareSymbolFont{mathdesignB}{MDB}{mdbch}{m}{n}%
\SetSymbolFont{mathdesignB}{bold}{MDB}{mdbch}{b}{n}%
\DeclareMathSymbol{\upintclockwise}{\mathop}{mathdesignB}{128}
\DeclareMathSymbol{\upointclockwise}{\mathop}{mathdesignB}{130}
\DeclareMathSymbol{\upointctrclockwise}{\mathop}{mathdesignB}{132}
\DeclareMathSymbol{\upoiint}{\mathop}{mathdesignB}{134}
\DeclareMathSymbol{\upoiiint}{\mathop}{mathdesignB}{136}

\makeatletter
\newcommand{\upint}{\DOTSI\upintop\ilimits@}
\newcommand{\upoint}{\DOTSI\upointop\ilimits@}
\makeatother

\renewcommand{\int}{\upint}

\usepackage{stackengine}

  \def\TreeOneTwoTwo#1#2#3#4#5{\tikz[baseline = -0.3ex]{
		\draw[fill] (0,0) circle [radius=0.06];
		\draw[fill] (-0.2,0.35) circle [radius=0.06];
		\draw[fill] (0.2,0.35) circle [radius=0.06];
            \draw[fill] (-0.4,0.7) circle [radius=0.06];
            \draw[fill] (0,0.7) circle [radius=0.06];
		\draw (0,0) -- (-0.2,0.35);
		\draw (0,0) -- (0.2,0.35);
            \draw (-0.2,0.35) -- (-0.4, 0.7);
            \draw (-0.2,0.35) -- (0, 0.7);
		\node[below] at (0,0) {$\scriptstyle{#1}$};
            \node[left] at (-0.2,0.35) {$\scriptstyle{#2}$};
            \node[left] at (-0.4,0.7) {$\scriptstyle{#3}$};
            \node[right] at (0,0.7) {$\scriptstyle{#4}$};
            \node[right] at (0.2,0.35) {$\scriptstyle{#5}$};
  }
  }
    \def\SmallTreeOneOne#1#2{\tikz[baseline = -0.3ex]{
		\draw[fill] (0,0) circle [radius=0.06];
		\draw[fill] (0,0.3) circle [radius=0.06];
		\draw (0,0) -- (0,0.3);
		\node[right] at (0,0) {$\scriptstyle{#1}$};
            \node[right] at (0,0.3) {$\scriptstyle{#2}$};
  }
  }

  \def\TreeOneThreeTwoOne#1#2#3#4#5#6#7{\tikz[baseline = -0.3ex]{
		\draw[fill] (0,0) circle [radius=0.06];
		\draw[fill] (-0.35,0.4) circle [radius=0.06];
            \draw[fill] (0,0.4) circle [radius=0.06];
		\draw[fill] (0.35,0.4) circle [radius=0.06];
            \draw[fill] (0,0.8) circle [radius=0.06];
            \draw[fill] (0.7,0.8) circle [radius=0.06];
            \draw[fill] (0.7,1.2) circle [radius=0.06];
		\draw[imperialGreen] (0,0) -- (-0.35,0.4);
		\draw[imperialBlue] (0,0) -- (0,0.4);
            \draw[imperialRed] (0,0) -- (0.35,0.4);
            \draw[imperialGreen] (0.35,0.4) -- (0,0.8);
            \draw[imperialBlue] (0.35,0.4) -- (0.7,0.8);
            \draw[imperialGreen] (0.7,0.8) -- (0.7,1.2);
		\node[below] at (0,0) {$\scriptstyle{#1}$};
            \node[left] at (-0.35,0.4) {$\scriptstyle{#2}$};
            \node[left] at (0.05,0.4) {$\scriptstyle{#3}$};
            \node[right] at (0.35,0.4) {$\scriptstyle{#4}$};
            \node[left] at (0,0.8) {$\scriptstyle{#5}$};
            \node[right] at (0.7,0.8) {$\scriptstyle{#6}$};
            \node[right] at (0.7,1.2) {$\scriptstyle{#7}$};}}

\usepackage{authblk}

\title{\textsc{Free probability, path developments and signature kernels as universal scaling limits}}

\author{Thomas Cass\textsuperscript{1,2}, William F. Turner\textsuperscript{1}}

\affil{\small\textsuperscript{1} Department of Mathematics, Imperial College London\\ \textsuperscript{2} Institute for Advanced Study, United States of America}

\date{\today}

\usepackage[style=alphabetic, giveninits = true, maxbibnames = 99]{biblatex}
\usepackage{csquotes}
\begin{document}

\maketitle

\begin{abstract}
    Random developments of a path into a matrix Lie group $G_N$ have recently been used to construct signature-based kernels on path space. Two examples include developments into GL$(N;\R)$ and $U(N;\C)$, the general linear and unitary groups of dimension $N$. For the former, \cite{SKlimit} showed that the signature kernel is obtained via a scaling limit of developments with Gaussian vector fields. The second instance was used in \cite{PCFGAN} to construct a metric between probability measures on path space. We present a unified treatment to obtaining large $N$ limits by leveraging the tools of free probability theory. An important conclusion is that the limiting kernels, while dependent on the choice of Lie group, are nonetheless universal limits with respect to how the development map is randomised. For unitary developments, the limiting kernel is given by the contraction of a signature against the monomials of freely independent semicircular random variables. Using the Schwinger-Dyson equations, we show that this kernel can be obtained by solving a novel quadratic functional equation. We provide a convergent numerical scheme for this equation, together with rates, which does not require computation of signatures themselves.
\end{abstract}

\section{Introduction}
In recent years the signature transform \cite{Chen54, Lyons_SF} has been used in a variety of applications, both as a way of vectorising sequential data, often as a step inside some machine learning pipeline \cite{LLN, Handwriting, DeepSig, SigSkel, ImageSIG}, and also as a tool to discriminate between probability measures on path space \cite{CO, SigWasserstein}.

The signature, which we will
denote by $\SS(\gamma)$, of a continuous path of bounded variation $\gamma$ in a vector space $V$ over an interval $[a,b]$ may be defined in a multitude of ways. One such approach is to define $\SS(\gamma)  :=z_{b},$ where $z$ solves the controlled differential equation
\begin{equation}\label{eq: sig}
\dif z_{t}=z_{t}\otimes \dif\gamma_t\text{ on }\left[  a,b\right]  \text{
}
\end{equation}
in the space $T((V)),$ the algebra of formal
tensor series over $V,$ with the initial condition $z_{a}=\mathbf{e,}$ the
identity in $T((V)).$ The structure of \cref{eq: sig} bears resemblance to the classical operation of
exponentiation, and the signature transform itself can loosely be interpreted
as an analogue of this operation with respect to the non-commutative product
$\otimes$. By selecting a basis for $V$, the signature can be described by a
collection of coefficients which determine the constituent tensors of
$\SS(\gamma)$. A key insight is that the ensemble of these
coefficients defines an algebra of functions on path space. They thus
form a basis in which functional relationships can be described, the canonical example of which being the map $\left(  y_{a},\gamma\right)
\mathbf{\mapsto}y,$ where $y$ solves the CDE%
\[
\dif y_{t}=V\left(  y_{t}\right)  \dif\gamma_{t}\,\text{started at }y_{a}.
\]
The solution to the specific \cref{eq: sig} can therefore encapsulate a
wide class of mappings of the form $\left(  y_{a},\gamma\right)
\mathbf{\mapsto}y.$ The signature allows one to reduce the analysis of paths $\gamma:\left[
a,b\right]  \rightarrow V$ to the analysis of a collection of tensors in $%
{\textstyle\prod\nolimits_{m=0}^{\infty}}
V^{\otimes m}$ associated with the signature of $\gamma$. The dimension of the
spaces $V^{\otimes m}$ increases exponentially in $m$, so that considerations
of dimensionality remain relevant. A commonly-taken approach is to truncate
the signature so that one deals instead with tensors in the more
computationally-manageable space ${\textstyle\prod\nolimits_{m=0}^{L}}V^{\otimes m}$. In other contexts, signature kernel methods have become
popular \cite{KO, SigPDE, issa2023nonadversarial, WSK}. These allow for efficient computation of inner-products of the form%
\[
K(\gamma,\sigma)\coloneqq \langle \SS(\gamma),\SS(\sigma)\rangle,%
\]
for various choices of inner product $\langle\cdot,\cdot\rangle$ on $T((V))$, without the need for computing $\SS(\gamma)$ directly.

One promising recent alternative to these ideas is a proposal to use the development of $\gamma$ into a matrix Lie group  \cite{lou2022path, PCFGAN}. In this formulation, the central
object becomes not the signature but related representations obtained by the
evolution
\begin{equation}\label{eq: dev 1}
\dif Z_{t}=M\left(  \dif\gamma_{t}\right)  \left(  Z_{t}\right)\,,Z_{a}%
=I_N,%
\end{equation}
in which $Z$ belongs to a matrix Lie group $G_{N}$, a subset of $\mathcal{M}%
_{N}\left(  F\right)  $, the space of $N\times N$ matrices over a field $F$,
and $M$ is a map in Hom$\left(  V,\text{Lie}\left(  G_{N}\right)  \right)  ,$
Lie$\left(  G_{N}\right)  $ denoting the space of of left-invariant vector
fields on $G_{N}.$ By the canonical identification of the Lie algebra
$\mathfrak{g}_{N}$ with Lie$\left(  G_{N}\right)  $ we can equivalently regard
$M$ as belonging to Hom$\left(  V,\mathfrak{g}_{N}\right)  $ and instead write
\cref{eq: dev 1} as
\begin{equation}\label{eq: dev 2}
\dif Z_{t}=Z_{t}\cdot M(\dif\gamma_{t})\,,Z_{a}=I_N%
\end{equation}
in which $\cdot$ denotes matrix multiplication. Important examples of this
framework include:
\begin{enumerate}[label=\arabic*)]
\item $F=\R\text{ or }\C$, and $G_{N}=$GL$\left(N;F\right)  $ the general
linear group for which $\mathfrak{g}_{N}=$ $\mathcal{M}_{N}\left(
F\right)$. The resulting features are denoted by $\GG(\gamma;M):=Z_b$.

\item $F=\mathbb{C}$, and $G_{N}=$U$\left(  N;\mathbb{C}\right)  $ the unitary group, whose corresponding Lie algebra is the set of anti-Hermitian matrices $\mathfrak{u}_{N}:=$ $\left\{  M\in\mathcal{M}_{N}\left(
\mathbb{C}\right)  :U^{\ast}=-U\right\}$. We denote the resulting unitary
features by $\mathcal{U}\left( \gamma;M\right)  :=Z_{b}.$
\end{enumerate}
A special case of the second example is when $N=1$, where the first unitary group is the unit circle
$U\left(  1;\C\right)  \cong\mathbb{T}:=\left\{  z\in%
\mathbb{C}
:\left\vert z\right\vert =1\right\},$ with complex
multiplication. In this case if $M$ is assumed to have the form $M=i\cdot f$
where $i=\sqrt{-1}$ and $f$ is in the dual space $V^{\ast}$ of the (real)
vector space $V,$ then we have that%
\[
Z_{b}=\exp\left(  if\left(  \gamma_{a,b}\right)  \right)  \text{ with }%
\gamma_{a,b}:=\gamma_{b}-\gamma_{a}.
\]
In particular if $\gamma_{a,b}$ is a random variable on $V$ with law $\mu$ then
$\mathbb{E}_{\mu}\left[  Z_{b}\right]  =\Phi_{\mu}\left(  f\right),$ the
classical characteristic function of $\mu$ evaluated at $f$. The development \cref{eq: dev 2} in $U(N;\C)$, was  also studied for general $N$ in the context of characterising probability measures on path space in \cite{CL}. These observations prompt Lou, Li and Ni in \cite{PCFGAN} to introduce a notion of distance between probability measures on path space first by defining
\[
\Phi_{\mu}(M):=\Ex_{\gamma\sim\mu}\big[\UU(\gamma; M)\big]\quad\text{and}\quad d_{M}\left(  \mu,\nu\right)  :=
\norm{\Phi_{\mu}(M) -\Phi_{\nu}(M)}_{\text{HS}}
\]
and then randomising the choice of development and setting%
\begin{equation}\label{eq: dist}
d_{\xi}\left(  \mu,\nu\right)  =\mathbb{E}_{M\sim\xi}\left[
d_{M}\left(  \mu,\nu\right)^2 \right]^{\frac{1}{2}}.%
\end{equation}
We remark that when when $N=1$, $d_\xi$ reduces to the characteristic function distance introduced in \cite{Ansari2019, CFGAN}. In the experiments of \cite{PCFGAN} the measure $\xi$ is
assumed to have the form $\xi=\frac{1}{k}\sum_{i=1}^{k}\delta_{M_{i}}$. The
collection of linear maps $\left\{  M_{1},...,M_{k}\right\}$ defining the
support of this measure are then treated a set of trainable parameters. The
function $d_{\xi}$ is used as a so-called discriminator, which is
then implemented within a generative adverserial network (GAN). The support of
$\xi$ is chosen to maximise the distance between a target distribution and a
proposed method of generating samples from the target distribution.
Results in the paper \cite{PCFGAN} assert state-of-the-art performance
in being able to generate time series samples arising from certain distributions.

Despite these promising results, questions remain which affect the wider use
of this approach. The critical issue of dimensionality outlined above, for
example, is transferred -- rather than resolved -- from the selection of a
truncation parameter in the signature, to the choice of a Lie group of
sufficiently high dimension for the method to be effective. The
overhead associated with computing $d_\xi$, as well as the difficulty of optimising the large parameter set forming
the discriminator remains and may, in some cases, be prohibitive. It is
desirable therefore to have some direct means of computing \cref{eq: dist},
especially in the case where $N$ may be large.

One way to tackle this problem is to observe \cite[Section B.3]{PCFGAN}, that $d_\xi$ is a maximum mean discrepancy distance (MMD). Indeed
\begin{equation*}
    d^2_{\xi}\left(  \mu,\nu\right)=\Ex_{M\sim\xi}\big[\langle\Phi_{\mu}(M),\Phi_{\mu}(M)\rangle\big]+\Ex_{M\sim\xi}\big[\langle\Phi_{\nu}(M),\Phi_{\nu}(M)\rangle\big]-2\Ex_{M\sim\xi}\big[\langle\Phi_{\mu}(M),\Phi_{\nu}(M)\rangle\big],
\end{equation*}
so that $d_\xi$ is an MMD with associated kernel
\[
\kappa(\gamma,\sigma):=\Ex_{M\sim\xi}\big[\langle \UU(\gamma;M),\UU(\sigma;M)\rangle_{\text{HS}}\big].
\]
As our interest is primarily in the asymptotic regime $N\rightarrow\infty$,
this leads us to ask whether we can characterise the large $N$ behaviour of
the law of the random variables
\begin{equation}\label{eq: rv}
\langle \mathcal{U}\left(  \gamma;\cdot\right)  ,\mathcal{U}\left(  \sigma;\cdot\right)  \rangle _{\text{HS}}\text{ under }\xi_{N}%
\end{equation}
for a suitable class of (Borel) probability measures $\xi_{N}$ on
$\mathfrak{u}_N$. While this may appear a
challenging technical problem, the theory of non-commutative probability \cite{VoiOrig, VoiMult, VoiMatProd} in fact provides a general framework both for the study of these limits, as well as giving explicit ways to characterise and calculate them. We recall that a non-commutative $C^\ast$-probability space is a pair $\left(
\AA,\phi\right)$ consisting of a $C^\ast$-algebra $\AA$ having a unit $\mathbf{e}$, and a
state (in the sense of \cite[Definition 5.2.0]{RMintro}) $\phi:\AA\rightarrow\mathbb{C}$ with $\phi\left(  \mathbf{e}\right)
=1$. Two well-known examples are

\begin{enumerate}[label=\arabic*)]
\item $\AA_{1}=$ $\mathcal{M}_{N}\left(  \mathbb{C}\right)$, and $\phi\left(
M\right)  =\frac{1}{N}$tr$\left(  M\right)$, .

\item $\AA_{2}=L^{\infty}\left(  \Omega;\mathcal{M}_{N}\left(  \mathbb{C}%
\right)  \right)$, the space of bounded $\mathcal{M}_{N}\left(
\mathbb{C}\right)  $-valued random variables, with the linear-functional $\phi\left(  M\right)
=\frac{1}{N}\mathbb{E}\left[  \text{tr}\left(  M\right)  \right]  $.
\end{enumerate}

One attractive feature of this theory is the way in which it can be used to
identify limits, both almost surely and in expectation, of sequences such as
\cref{eq: rv}. These limits are typically determined by the dependencies between the matrix entries but are otherwise universal. They do not restrict
$\xi_{N}$ to belong to a class of nicely-described distributions, in
particular they need not be Gaussian, and need instead only soft moment
assumptions. The archetype of this result shows that sequences (of
collections) of random Wigner matrices $\mathcal{L}_N:=\Big\{  \frac
{1}{\sqrt{N}}M_{i}^{N}\Big\}  _{i=1}^{d}\in\mathcal{M}_{N}\left(
\mathbb{C}\right)  $, under very general assumptions, are asymptotically free
and converge in law (in the sense of non-commutative probability) to the
semicircle law. These statements hold both in the almost sure sense, viewing
$\mathcal{L}_{N}$ as a random set in $\big(  \mathcal{M}_{N}(
\mathbb{C})  \text{,}\frac{1}{N}\text{tr}(  \cdot)  \big)  ,$ and
in expectation, by working instead with $\phi\left(  \cdot\right)  =\frac
{1}{N}\mathbb{E}\left[  \text{tr}\left(  \cdot\right)  \right]$.

By using techniques from this theory we are able to define the limit%
\[
K_{\text{SD}}\,\left(  \gamma,\sigma\right):=\lim_{N\to
\infty}\frac{1}{N}\mathbb{E}_{M\sim\xi_{N}}\left[  \langle \mathcal{U}\left(
\gamma;M\right)  ,\mathcal{U}\left(  \sigma;M\right)  \rangle
_{\text{HS}}\right]  ,
\]
and moreover understand $K_{\text{SD}}\,\left(  \gamma,\sigma\right)  $ in
terms of the signature of the path $y=\gamma\ast\overleftarrow{\sigma}$%
\begin{equation}\label{SD}
K_{\text{SD}}\,\left(  \gamma,\sigma\right)  =\sum_{k=0}^{\infty}\left(
-1\right)  ^{k}\sum_{\left\vert \bm{I}\right\vert =k}\varphi\left(  X_{\bm{I}}\right)
\SS^{\bm{I}}\left(  y\right) ,%
\end{equation}
where $\left\{  X_{1},...,X_{d}\right\}  $ are free semicircular random
variables with $\varphi\left(  X_{\bm{I}}\right)  $ the non-commutative moment of the
monomial $X_{\bm{I}_{1}}...X_{\bm{I}_{k}}.$ To make practical use of this result, we
need some means of computing \cref{SD} that avoids the direct computation of
signatures, by analogy to the way the signature kernel PDE can be solved to
determine the signature kernel \cite{SigPDE}. The structure of free semicircular variables
offers a way to do this. One way to characterise their properties is through
the \textit{Schwinger-Dyson equations }(hence $K_{\text{SD}}$). These
equations, which originate in quantum field theory, have an important
realisation in the context of non-commutative probability. We use this
structure to prove that $K_{\text{SD}}\,\left(  \gamma,\sigma\right)
=K_{\text{SD}}^{y}\left(  a,b\right)  $, where $y=\gamma\ast\overleftarrow{\sigma}$ and $K_{\text{SD}}^{y}:\left\{
a\leq s\leq t\leq b\right\}  \rightarrow%
\mathbb{R}
$ solves the quadratic functional equation
\begin{equation}\label{eq: SD func}
\begin{split}
    K_{\text{SD}}^{y}\left(  s,t\right)   & =1-\int_{a}^{t}\int_{a}^{s}%
K_{\text{SD}}^{y}\left(  a,u\right)  K_{\text{SD}}^{y}\left(  u,r\right)
\langle \dif y_{u},\dif y_{r}\rangle ,\text{ }\\
K_{\text{SD}}^{y}\left(  s,s\right)   & =1\text{ for all }s\in\left[
a,b\right].
\end{split}
\end{equation}
We show that this equation, which appears not to have been previously studied,
has a unique solution. We then propose and implement a numerical scheme to
solve \cref{eq: SD func} and establish convergence of the scheme to the true
solution together with convergence rates.

As a further application of the use of these methods we consider the evolution
\cref{eq: dev 2} in $G_{N}=$GL$\left(N;\mathbb{C}\right)$ the general linear group. The corresponding limit under these
assumptions is the classical signature kernel:%

\[
K_{\text{sig}}\left(  \gamma,\sigma\right)  =\lim_{N\rightarrow
\infty}\frac{1}{N}\mathbb{E}_{M\sim\xi_{N}}\left[  \langle \mathcal{G}\left(
\gamma; M\right)  ,\mathcal{G}\left(\sigma; M\right)  \rangle
_{\text{HS}}\right]  .
\]
For the specific case where $\xi_{N}$ is the Gaussian measure on Hom$(V,\mathcal{M}_{N}\left(  \mathbb{R}\right))  $ with variance $\frac{1}{N}$, this
result was obtained in the recent paper of Mu\c{c}a Cirone et al. \cite[Theorem 4.4]{SKlimit}. By adopting the
techniques outlined above, we are able to give a concise proof that
also covers the non-Gaussian case, allows for complex valued matrices and dispenses with some unnecessary
independence conditions. We note that \cite{SKlimit} also derives limits for non-linear randomised CDEs, though the proof techniques differ substantially from the linear case. The original motivation in \cite{SKlimit} is to analyse the scaling limits of controlled ResNets arising as the Euler discretisation of CDEs; in particular the authors focus on the kernel function of the limiting random processes. We hope our results help broaden this emerging connection between between families of kernels on unparameterised path space and scaling limits of controlled ResNets, where infinite width and depth limits are known to commute in certain cases \cite{SKlimit}.

This article also sits within a broader body of literature surrounding the scaling limits of random neural networks and the application of random matrix theory within machine learning. Random matrix theory and free probability theory are swiftly arising as popular tools to analyse the properties of relevant matrices. For example, in \cite{pennington2017resurrecting}, the authors employ techniques from free probability theory to derive the singular value distribution of the input-output Jacobian of a network. Furthermore, the eigenvalue distribution of certain non-linear random matrix models was studied in \cite{Eval_nonlinear} and the limiting spectrum of the neural tangent kernel (NTK) has also received much attention \cite{NTK_RMT,wang2023deformed}. We refer the reader to \cite{RMT_DL_summary} for a survey and references on recent results in this area.
\subsection{Outline}
In \cref{sec: Sk} we leverage techniques from non-commutative probability theory to provide a universal version of \cite[Theorem 4.4]{SKlimit}. We adapt the methods from the proof of \cite[Theorem 5.4.2]{RMintro}, which deals with Wigner random matrices, to the setting of $\MM_N(\C)$-valued random matrices. The core of the approach in \cite{RMintro} is the association of undirected graphs with words from an alphabet of $N$ letters. These words then correspond to the individual terms that form the trace of a product of random matrices. When dealing with generic $\MM_N(\C)$-valued matrices, we additionally associate a directed graph with each word. This insight, when combined with the approach of Mu\c{c}a Cirone et al., yields a more concise proof which eliminates certain Gaussian and independence assumptions, and allows for complex valued matrices. The relevant definitions on words and their associated graphs are found in \cref{sec: letters}.

We present our main result \cref{thm: fpcf_eq} in \cref{sec: main_result}, where we prove that a universal limit exists for random $U(N;\C)$-valued path developments. The existence of such a limit is a relatively straightforward consequence of \cite[Theorem 5.4.2]{RMintro} and the techniques from \cref{sec: Sk}. Our main contribution, however, is showing that this limit solves an explicit quadratic functional equation, which has a unique solution, see \cref{lem: func_unique}. The proof of uniqueness is non-trivial due to the quadratic nature of the functional equation. Our proof relies on Dyck words and, to the best of our knowledge, a new notion of the ``generation'' of a Dyck word. The generation of a Dyck word is not its length, but the number of steps taken to generate the word using a certain algorithm; see \cref{sec: generations} for details.

Finally, in \cref{sec: numerical}, we outline an explicit numerical scheme to solve the functional equation. We prove convergence along with rates in \cref{thm: numerical_conv} through an object known as the iterated sums signature and corresponding results from \cite{KO}. We end by comparing this numerical scheme with the Monte Carlo approach of averaging multiple random $U(N;\C)$ path developments with large $N$. In particular, we analyse how the computational time for each method scales in the length of the input path, the dimension of the input path and the dimension $N$ in the case of the randomised approach.

\subsection{Letters, Words and Graphs}\label{sec: letters}
For ease of cross referencing we will base our notation on \cite{RMintro}. Given an alphabet $\AA_n$ of $n$ letters, a word $\bm{w}=w_1w_2\dots w_k\in\WW_n$ is a sequence of letters in $\AA_n$. For $n\geq 1$, let $\WW_n^k$ be the set of words in $n$ letters of length $k$. We also define the set $\WW_n^0=\{\varnothing\}$, which contains the empty word. For a word $\bm{w}=w_1\dots w_k\in\WW_n^k$ and letter $i\in\AA_n$, we define the word $\bm{w}i\coloneqq w_1\dots w_k i\in\WW_n^{k+1}$. Furthermore, for any two words $\bm{w}\in\WW_n^k$ and $\bm{u}\in\WW_n^l$, we define their concatenation as the word $\bm{w}\bm{u}\coloneqq w_1\dots w_k u_1\dots u_l\in\WW_n^{k+l}$. For a word $\bm{w}=w_1\dots w_k\in\WW_n^k$, we define
\begin{enumerate}[label=\arabic*)]
    \item Its length as $\vert\bm{w}\vert=k$.
    \item The support of $\bm{w}$, $\text{supp}(\bm{w})\subseteq \AA_n$, as the set of distinct letters in $\bm{w}$.
    \item Its weight, $\wt{\bm{w}}$, to be the number of distinct letters in $\bm{w}$.
    \item Its transpose $\bm{w}^\star\coloneqq w_k\dots w_1$;
    \item An undirected graph ${G}_{\bm{w}}^u=\big({V}_{\bm{w}},{E}_{\bm{w}}^u\big)$ associated with $\bm{w}$ with vertex set $V_{\bm{w}}= \text{supp}(\bm{w})$ and (undirected) edge set
    \begin{equation*}
        E_{\bm{w}}^u = \big\{\{w_i,w_{i+1}\}: i=1,\dots,k -1\big\}.
    \end{equation*}
    \item A directed graph ${G}_{\bm{w}}^d=\big({V}_{\bm{w}},{E}_{\bm{w}}^d\big)$ associated with $\bm{w}$ with vertex set $V_{\bm{w}}= \text{supp}(\bm{w})$ and (directed) edge set
    \begin{equation*}
        E_{\bm{w}}^d = \big\{(w_i,w_{i+1}): i=1,\dots,k -1\big\}.
    \end{equation*}
    \end{enumerate}
We recall that a walk on a graph $G=(E,V)$ is a sequence of edges connecting a sequence of vertices and a path is a walk in which all vertices are distinct. Finally, a traversal is a walk that visits each vertex at least once. Any word $\bm{w}$ defines a traversal on its associated graph ${G}_{\bm{w}}^u$ and for any edge $e\in E_{\bm{w}}^u,$ we set $N_{\bm{w}}^e$ as the number of times $e$ is traversed in the traversal defined by $\bm{w}$. We call a word $\bm{w}=w_1\dots w_k$ \textit{closed} if $w_1=w_k$. We call a closed word $\bm{w}\in\WW_n^{k+1}$ a \textit{Wigner} word if $\wt{w}=\frac{k}{2}+1$ and $N_{\bm{w}}^e\geq 2$ for each $e\in E_{\bm{w}}^u$. We note that by \cite{RMintro}, this in fact implies $N_{\bm{w}}^e=2$ for any edge in the undirected graph associated with a Wigner word. Finally, two words $\bm{w},\bm{u}\in\WW_n$ are called equivalent if there is a bijection on $\AA_n$ that maps the words onto each other. 

\subsection{The Signature}
Throughout, let $V\equiv\R^d$ be a finite dimensional vector space equipped with an inner product $\langle \cdot, \cdot \rangle$ and associate norm $\norm{\cdot}.$ We recall the notion of the $1$-variation semi-norm.
\begin{definition}
    We write $\XX\coloneqq\CC_1([0,T],V)$ for the space of continuous paths with $\gamma_0=0$ and finite $1$-variation in the sense that
    \begin{equation*}
        \norm{\gamma}_1\coloneqq \sup_{\PP=\{t_i\}}\sum_{i}\big\|{\gamma_{t_{i+1}}-\gamma_{t_i}}\big\|<\infty.
    \end{equation*}
\end{definition}
The signature of a path will be a central object in this work.
\begin{definition}
    For $\gamma\in\XX$, the signature of $\gamma$ is a function
    \[
    \SS_{\cdot,\cdot}(\gamma):[0,T]^2\cap\{(s,t):0\leq s\leq t\leq T\}\to T((V))\coloneqq \prod_{m=0}^\infty V^{\otimes m},
    \]
    where $\SS_{s,s}(\gamma) = (1,0,\dots)$ for all $s\in [0,T]$ and
    \[
    \SS_{s,t}(\gamma)^m\coloneqq \int_{0\leq t_1\leq \dots \leq t_m\leq 1}\dif\gamma_{t_1}\otimes\dots\otimes \dif \gamma_{t_m}\in V^{\otimes m},
    \]
    for $m\geq 1$ with $\SS_{s,t}(\gamma)^0\coloneqq 1\in\R$. For a word $\bm{I}\in\WW_d^m$, we also define the coordinate iterated integral
    \[
    \SS_{s,t}^{\bm{I}}(\gamma)\coloneqq \int_{0\leq t_1\leq \dots \leq t_m\leq 1}\dif\gamma_{t_1}^{\bm{I}_1}\otimes\dots\otimes \dif \gamma_{t_m}^{\bm{I}_m}\in \R,
    \]
    where $\SS_{s,t}^{\varnothing}(\gamma)\coloneqq \SS_{s,t}(\gamma)^0=1.$
\end{definition}
The signature kernel, originally introduced by Kir\`aly and Oberhauser in \cite{KO}, is defined in the following way.
\begin{definition}
    The signature kernel between two paths $\gamma,\sigma\in\XX$ is defined by
    \begin{equation}
        K_{\text{sig}}^{\gamma,\sigma}(s,t)\coloneqq \langle \SS_{0,s}(\gamma),\SS_{0,t}(\sigma)\rangle := \sum_{m=0}^\infty \langle \SS_{0,s}(\gamma)^m,\\S_{0,t}(\sigma)^m\rangle_{\text{HS}},
    \end{equation}
    where for each $m$, $\langle\cdot,\cdot\rangle_{\text{HS}}$ denotes the Hilbert-Schmidt inner product on $V^{\otimes m}$.
\end{definition}
Efficient computation of the signature kernel has seen much interest lately. The authors of \cite{KO} provided a kernel trick for a truncated version of $K_{\text{sig}}(s,t)$. In \cite{SigPDE} it is shown that when $K_{\text{sig}}$ satisfies a Goursat PDE in the case that $\gamma$ and $\sigma$ are continuously differentiable. Finally, in \cite{SKlimit} a connection between the signature kernel and $N$-dimensional randomised controlled differential equations (CDEs) was established in the limit as $N\to\infty$.

\section{The Signature Kernel is a Universal Limit}\label{sec: Sk}
We consider first the case of GL$(N;\R)$-developments. In \cite{SKlimit}, it was shown that the signature kernel between $\gamma,\sigma\in\XX$ may be recovered from the scaling limit of these random developments driven by $\gamma$ and $\sigma$. Indeed, if with $\xi_N$ is the the Gaussian measure on Hom$(V, \MM_N(\R))$ with variance $\frac{1}{N}$, then one obtains
\begin{equation}\label{eq: sk_as_dev}
K_{\text{sig}}^{\gamma,\sigma}(T, T)  =\lim_{N\rightarrow
\infty}\frac{1}{N}\mathbb{E}_{M\sim\xi_{N}}\left[  \langle \mathcal{G}\left(
\gamma; M\right)  ,\mathcal{G}\left(\sigma; M\right)  \rangle
_{\text{HS}}\right].
\end{equation}
The proof of \cref{eq: sk_as_dev} in \cite{SKlimit} is specialised to the Gaussian setting and relies heavily on Isserli's Theorem. We identify a random matrix theory approach, inspired by \cite{RMintro}, where we generalise the preceding in the following ways.
\begin{enumerate}[label=\arabic*)]
    \item We replace the space GL$(N;\R)$ with GL$(N;\C)$, whose Lie algebra is $\mathfrak{g}_N=\MM_N(\C)$.
    \item We replace the Gaussian assumption on the entries of the matrices with a sub-Gaussian assumption on the operator norm of the matrices.
    \item For fixed indices $(m,l)\in\{1,\dots,N\}^2$ we require that $A_i^{N}(m,l)$ and $A_j^N(m,l)$ are uncorrelated and not necessarily independent.
\end{enumerate}
Moreover, our approach results in shorter proof than \cite{SKlimit} and allows for different limits to be derived for randomised CDEs with a different matrix structure, see \cref{sec: main_result}. In the following we will use $\norm{\cdot}_{\mathrm{op}}$ to denote the operator norm of matrices.

If one fixes a basis for $V$, then every $M\in\text{Hom}(V,\mathfrak{g}_N)$ may be written in the form
\[
M(v)=\sum_{j=1}^dA_jv^j,\quad\text{for }A_i\in\mathfrak{g}_N,
\]
so that we have a canonical isomorphism between Hom$(V,\mathfrak{g}_N)$ and $\mathfrak{g}_N^d$. In particular, the spaces of Borel probability measures $\PP(\text{Hom}(V,\mathfrak{g}_N))$ and $\PP(\mathfrak{g}_{N}^d)$ are in bijection. Additionally, the development in \cref{eq: dev 2} can then be written as
\[
\dif Z_t=\sum_{i=1}^dZ_tA_i\dif\gamma_t^i,\quad Z_0=I_N.
\]
\begin{theorem}\label{thm: main_thm}
    Let $(\Omega, \FF, \P)$ be a probability space supporting for each $N\in\N$ and $1\leq i \leq d$, a random matrix $A_i^N:\Omega \to \mathfrak{g}_N$. Assume that for all $k\in\N$,
    \begin{equation}\label{eq: moment_condition}
        \sup_{N\in\N}\sup_{1\leq i \leq d}\sup_{1\leq m,l\leq N}\Ex\big[\big\vert{A_i^N(m,l)}\big\vert^k\big]\leq c_A(k)< \infty,
    \end{equation}
    and that $\Ex\big[A_i^N(m,l)\big]=0$ and $\Ex\big[\big\vert A_i^N(m,l)\big\vert^2\big]=1$. Suppose also that there exists some $\kappa>0$ such that for every $p\in\N$
    \begin{equation}\label{eq: sub_gaussian_condition}
        \sup_{N\in\N}\sup_{1\leq i \leq d}\Ex\bigg[\frac{1}{N^{\frac{p}{2}}}\normopp[\big]{A_i^N}\bigg]\leq \kappa^p\Gamma(\tfrac{p}{2}+1),
    \end{equation}
    where $\Gamma(\cdot)$ is the Gamma function. For each $N\in\N$, we assume that
    \begin{equation}
        \bigvee_{i=1}^d\sigma\big(A_i^N(m,l\big),\quad \text{for } 1\leq m,l\leq N,
    \end{equation}
    are independent sigma algebras and that
    \begin{equation}
    \Ex\big[A_i^N(m,l)A_j^N(m,l)\big]=0,\quad\text{for }i\neq j.
    \end{equation}
    For each $N$, let $\xi_N$ be the law of this collection in $\PP(\mathfrak{g}_N^d)$. Let $\gamma, \sigma\in\XX$ and let $Z_\gamma^N(t),Z_\sigma^N(t)$ be the $\textrm{GL}(N;\C)$-valued solutions to
    \begin{align}\label{eq: CDEs}
        \dif Z_\gamma^N(t) &= \frac{1}{\sqrt{N}}\sum_{i=1}^d Z_\gamma^N(t)A_i^N\dif\gamma_t^i,\ Z_\gamma^N(0)=I_N\\
        \dif Z_\sigma^N(t) &= \frac{1}{\sqrt{N}}\sum_{i=1}^d Z_\sigma^N(t)A_i^N\dif\sigma_t^i,\ Z_\sigma^N(0)=I_N.
    \end{align}
    Then for every $s,t\in[0,T]$
    \begin{equation}
        \lim_{N\to\infty} \frac{1}{N}\Ex_{\xi_N}\big[\langle Z_\gamma^N(s),Z_\sigma^N(t)\rangle_{\text{HS}}\big]=K^{\gamma,\sigma}_{\text{sig}}(s,t).
    \end{equation}
\end{theorem}
As in \cite{SKlimit}, the proof relies on the expansion of \cref{eq: CDEs} as a convergent infinite series of integrals, and a careful analysis of the trace of products of random matrices from the set $\{A_i^N:1\leq i \leq d\}$ in the limit as $N\to \infty$.
\begin{remark}
    Special cases of the following lemma can be deduced from existing results in the literature. For example, both the complex and real Gaussian case may be derived from \cite[Sections 2 and 4.4]{RGinibre}, while matrices with i.i.d. entries satisfying a fourth moment condition can also be found in \cite[Section 4.1]{RGinibre}.
\end{remark}
\begin{notation}
    For $A\in\MM_N(\C)$, $A^\ast$ will denote the conjugate transpose of $A$. The Hilbert-Schmidt inner product on $\MM_N(\C)$ is then given by
    \[
    \langle A, B\rangle_{\text{HS}} = \text{tr}\big(A^\ast B\big).
    \]
\end{notation}
\begin{lemma}\label{lem: correct_limit}
    Suppose that for each $N\in\N$, $\{A_i^N:1\leq i \leq d\}$ is a collection of $d$ matrices satisfying the assumptions of \cref{thm: main_thm}. Let $n,m\geq 0$ be non-negative integers and let $\bm{I}=i_1\dots i_n \in \WW_d^n$, $\bm{J}=j_1\dots j_m \in \WW_d^m$, with corresponding matrix products 
    \begin{align*}
           {A_{\bm{I}}^N} &\coloneqq {A_{i_1}^N}\dots {A_{i_n}^N}\\
            A_{\bm{I}\star \bm{J}}^N &\coloneqq {A_{\bm{I}}^{N}}^\ast A_{\bm{J}}^N          
    \end{align*}
    Then, setting $k=n+m$,
    \begin{equation}\label{eq: correct_limit}
        \lim_{N\to\infty}\frac{1}{N^{\frac{k}{2}+1}}\Ex\Big[\tr{A_{\bm{I}\star \bm{J}}^N}\Big]=
        \begin{cases}
            1,\quad \text{if }\bm{I}=\bm{J},\\
            0,\quad \text{otherwise}.
        \end{cases}
    \end{equation}
\end{lemma}
We use the convention that if $n=0$, then $ A_{\bm{I}}^N=I_N$, the $N\times N$ identity matrix.
\begin{proof}
    We start by expanding the LHS of \cref{eq: correct_limit} in terms of the individual elements composing the trace. We write
    \begin{equation}\label{eq: trace_expansion}
        \frac{1}{N^{\frac{k}{2}+1}}\Ex\Big[\tr{A_{\bm{I}\star \bm{J}}^N}\Big]=\frac{1}{N^{\frac{k}{2}+1}}\sum_{\bm{w}}\mathscr{C}_{\bm{w}},
    \end{equation}
    where the sum is over closed words $\bm{w}=w_1 w_2\dots w_n w_{n+1} \dots w_k w_{1}\in\WW_{N}^{k+1}$ of length $k+1\coloneq n+m+1$, and where
    \begin{align*}
         \mathscr{C}_{\bm{w}}&\coloneqq \Ex\bigg[{A_{i_n}^{N^\ast}}(w_1,w_2) {A_{i_{n-1}}^{N^\ast}}(w_2,w_3)\dots {A_{i_1}^{N^\ast}}(w_n,w_{n+1}) {A_{j_1}^N}(w_{n+1},w_{n+2})\dots {A_{j_m}^N}(w_{k},w_{1}) \bigg]
    \end{align*}
    Since $\bm{w}$ is a closed word, it defines a traversal on the undirected graph ${G}_{\bm{w}}^u$. As in \cite{RMintro}[Theorem 5.4.2, Lemma 2.1.6], we can deduce that
    \begin{enumerate}[label=\arabic*)]
        \item $\mathscr{C}_{\bm{w}}$ vanishes unless $N_{\bm{w}}^e\geq 2$ for all $e\in {E}_{\bm{w}}^u$.
        \item The number of $N$-words in the equivalence class of any given $N$-word $\bm{w}$ with $\wt{\bm{w}}=t$ is $N(N-1)\dots (N-t+1)\leq N^t$. This can be seen by counting the size of the image of $\bm{w}$ under all bijections on $\AA_n$.
        \item The number of equivalence classes of closed $N$-words $\bm{w}$ of length $k+1$ with $\wt{\bm{w}}=t$ for which $N_{\bm{w}}^e\geq 2$ for every $e\in {E}_{\bm{w}}^u$ is bounded above by $t^k\leq k^k$, since there are at most $t$ choices for each letter excluding the final one.
        \item The weight of a closed word of length $k+1$ with $N_{\bm{w}}^e\geq 2$ for every $e\in {E}_{\bm{w}}^u$ is at most $\frac{k}{2}+1$.
    \end{enumerate}
    It follows that
    \[
    \left\vert \sum_{\bm{w}:\wt{\bm{w}}\leq \frac{k}{2}}\mathscr{C}_{\bm{w}}\right\vert\leq \sum_{t\leq \frac{k}{2}} N^t c_{A}(k)t^k\leq C(k)N^{\frac{k}{2}},
    \]
    and combining this with \cref{eq: trace_expansion}, we see 
    \[
    \left\vert\frac{1}{N^{\frac{k}{2}+1}}\Ex\Big[\tr{A_{\bm{I}\star \bm{J}}^N}\Big] - \frac{1}{N^{\frac{k}{2}+1}}\sum_{\bm{w}:\wt{\bm{w}}=\frac{k}{2}+1}\mathscr{C}_{\bm{w}}\right\vert\leq \frac{C(k)}{N}.
    \]
    Note that the set $\{\bm{w}:\wt{\bm{w}}=\frac{k}{2}+1\}$ is empty whenever $k$ is odd. Hence, if $n+m$ is odd, we have that
    \[
   \lim_{N\to\infty}\frac{1}{N^{\frac{k}{2}+1}}\Ex\Big[\tr{A_{\bm{I}\star \bm{J}}^N}\Big]=0.
    \]
    If $k$ is even and $\bm{w}$ is a closed word of length $k+1$ with $\wt{\bm{w}}=\frac{k}{2}+1$, then $G_{\bm{w}}^u$ is a tree \cite{RMintro}[Chapter 2.1.3] and that $N_{\bm{w}}^e=2$ for every $e\in {E}_{\bm{w}}^u$. As such, $\mathscr{C}_{\bm{w}}$ is a product of covariances which will either be $1$ or $0$. There are two situations in which this product can vanish. Firstly, whenever the two traversals of an edge $e\in {E}_{\bm{w}}^u$ correspond to different matrices. Secondly, if both traversals are from the same matrix, but one traversal corresponds to an entry above the diagonal and the other to an entry below the diagonal. It is the latter reason that differentiates our problem from the Hermitian setting.

    Since $G_{\bm{w}}^u$ is a tree, the traversal defined by $\bm{w}$ traverses each edge once in each direction. Let $(w_{r}, w_{r+1})\in E_{\bm{w}}^d$ correspond to the first traversal of $\{w_r,w_{r+1}\}$. Then, for some $s\in \{r+1,\dots k\}$, $w_s = w_{r+1}$ and $w_{s+1} = w_r$. In other words $(w_s, w_{s+1})=(w_{r+1},w_r)$. If $s\leq n $, the corresponding component of $\mathscr{C}_{\bm{w}}$ is given by
   \begin{equation*}
        \Ex\Big[{A_{i_{n+1-r}}^{N^\ast}}(w_r,w_{r+1}) {A_{i_{n+1-s}}^{N^\ast}}(w_s,w_{s+1})\Big]=\Ex\Big[{A_{i_{n+1-r}}^{N^\ast}}(w_r,w_{r+1}) {A_{i_{n+1-s}}^{N^\ast}}(w_{r+1},w_{r})\Big] = 0.
    \end{equation*}
    Similarly, if $r\geq n + 1$, the term is given by
     \begin{equation*}
        \Ex\Big[A_{j_{r-n}}^N(w_r,w_{r+1}) A_{j_{s-n}}(w_s,w_{s+1})\Big]=\Ex\Big[A_{j_{r-n}}^N(w_r,w_{r+1}) A_{j_{s-n}}^N(w_{r+1},w_{r})\Big] = 0.
    \end{equation*}
    So it must be that $r\leq n$ and $s \geq n+1$. Consequently, for every edge, one traversal comes from $\bm{I}$ and one from $\bm{J}$ meaning that $|\bm{I}|=|\bm{J}|$. Furthermore, since $G_{\bm{w}}^u$ is a tree, there can be no repeated letters in $w_1\dots w_{n+1}$, as otherwise there would be two traversals of an edge coming from $\bm{I}$. Hence, the walk corresponding to the sequence of vertices $w_1,\dots,w_{n+1}$ is a path, as is the walk on the sequence $w_{n+1},w_{n+2},\dots,w_1$. Consider the vertex $w_{n+1}$; we must traverse the edge $(w_{n+1},w_n)$ at some point on the path $w_{n+1},w_{n+2},\dots,w_1$. Since the path cannot visit $w_{n+1}$ again, in order to traverse the edge $(w_{n+1},w_n)$, it must be that $w_{n+2}=w_n$. By induction, the only way for each edge to be traversed twice is when $w_{n+1-r}=w_{n+1+r}$ for $r=1,\dots,n$. And so there is exactly one equivalence class of words for which $\mathscr{C}_{\bm{w}}$ is non-zero: those where $G_{\bm{w}}^u$ is a tree with one branch. When this holds, the term for the traversals of an edge $\{w_{n+1-r},w_{n+2 -r}\}$ in $\mathscr{C}_{\bm{w}}$ is
    \begin{align*}
    &\Ex\Big[{A_{i_{n+1-(n+1-r)}}^{N^\ast}}(w_{n+1-r},w_{n+1 -(r-1)}) {A_{j_{(n+1+(r-1)) - n}}^N}(w_{n+1+(r-1)},w_{n+1+r})\Big]\\
    =&\Ex\Big[\overline{{A_{i_{r}}^N}(w_{n+1 -(r-1)},w_{n+1-r})} {A_{j_{r}}^N}(w_{n+1+(r-1)},w_{n+1+r})\Big]=\delta_{{i_r}{j_r}}.
    \end{align*}
    This implies $\bm{I}=\bm{J}$ for the limit to be non-trivial. We conclude that
    \begin{align*}
            \lim_{N\to\infty} \frac{1}{N^{\frac{k}{2}+1}}\Ex\big[\tr{ A_{\bm{I}\star\bm{J}}^N}\big]&= \lim_{N\to\infty}\frac{1}{N^{\frac{k}{2}+1}}\sum_{\bm{w}:\wt{\bm{w}}=\frac{k}{2}+1}\mathscr{C}_{\bm{w}}\\
            &=\lim_{N\to\infty}\frac{N(N-1)\dots (N-\frac{k}{2})}{N^{\frac{k}{2}+1}}\delta_{\bm{I}\bm{J}} =\delta_{\bm{I}\bm{J}}.
    \end{align*}
    \end{proof}
    
\begin{proof}[Proof of \cref{thm: main_thm}]
    As in \cite{SKlimit}, by considering the expansion of the ordered-exponential we see that
    \begin{equation}\label{eq: sig_expansion}
         \lim_{N\to\infty}\frac{1}{N}\Ex_{\xi_N}\big[\langle Z_\gamma^N(s),Z_\sigma^N(t)\rangle_{\text{HS}}\big]=\lim_{N\to\infty}\Ex_{\xi_N}\bigg[\sum_{|\bm{I}|,|\bm{J}|=0}^\infty \frac{1}{N^{\frac{|\bm{I}^\star\bm{J}|}{2}+1}}\tr{ A_{\bm{I}\star \bm{J}}^N}\hspace{1mm}\SS_{0,s}^{\bm{I}}(\gamma)\SS_{0,t}^{\bm{J}}(\sigma)\bigg].
    \end{equation}
    In order to apply \cref{lem: correct_limit}, we would like to exchange the limit and expectation and the double sum. By Fubini-Tonelli and the dominated convergence theorem, to justify the exchange it is enough to show that
\begin{equation}\label{eq: ft_dct_bound}
    \sum_{|\bm{I}|,|\bm{J}|=0}^\infty \frac{1}{N^{\frac{|\bm{I}^\star\bm{J}|}{2}+1}}\Ex_{\xi_N}\Big[\Big\vert\tr{A_{\bm{I}\star \bm{J}}^N }\Big\vert\Big] \hspace{1mm}\big\vert\SS_{0,s}^{\bm{I}}(\gamma)\SS_{0,t}^{\bm{J}}(\sigma)\big\vert
\end{equation}
is uniformly bounded in $N$. By the assumptions of \cref{thm: main_thm},
\begin{equation*}
    \frac{1}{N^{\frac{|\bm{I}^\star\bm{J}|}{2}+1}}\Ex_{\xi_N}\Big[\Big\vert\tr{A_{\bm{I}\star \bm{J}}^N }\Big\vert\Big]\leq  \frac{1}{N^{\frac{|\bm{I}^\star\bm{J}|}{2}+1}}\Ex_{\xi_N}\Big[\normop[\big]{A_{\bm{I}\star \bm{J}}^N}\Big]\leq \kappa^{|\bm{I}^\star\bm{J}|}\Gamma\big(\tfrac{|\bm{I}^\star\bm{J}|}{2}+1\big).
\end{equation*}
Then, by the factorial decay of the signature and the fact that the $L_2$ norm is at least as large as the $L_1$ norm on $V^{\otimes n}$, there exists some $\omega>0$ for which \cref{eq: ft_dct_bound} is bounded by
\begin{align*}
    \sum_{n,m=0}^\infty \frac{\Gamma(\tfrac{n+m}{2}+1)(\omega\kappa)^{m+n}}{\Gamma(n+1)\Gamma(m+1)}&\leq \sum_{n,m=0}^\infty \frac{\sqrt{\Gamma(n+1)\Gamma(m+1)}(\omega\kappa)^{n+m}}{\Gamma(n+1)\Gamma(m+1)}\\
    &=\bigg(\sum_{n=0}^\infty\frac{(\omega\kappa)^{n}}{\sqrt{\Gamma(n+1)}}\bigg)^2<\infty,
\end{align*}
where the inequality follows from logarithmic convexity and Jensen's inequality. Hence, by an exchange of limits and an application of \cref{lem: correct_limit}, we see that
\begin{align*}
    \lim_{N\to\infty}\frac{1}{N}\Ex_{\xi_N}\big[\langle Z_\gamma^N(s),Z_\sigma^N(t)\rangle_{\text{HS}}\big]&=
    \sum_{|\bm{I}|,|\bm{J}|=0}^\infty\lim_{N\to\infty} \frac{1}{N^{\frac{|\bm{I}^\star\bm{J}|}{2}+1}}\Ex_{\xi_N}\Big[\tr{A_{\bm{I}\star \bm{J}}^N}\Big]\hspace{1mm}\SS_{0,s}^{\bm{I}}(\gamma)\SS_{0,t}^{\bm{J}}(\sigma)\\
    &=\sum_{|\bm{I}|=0}^\infty \SS_{0,s}^{\bm{I}}(\gamma)\SS_{0,t}^{\bm{I}}(\sigma)\\
    &=\langle \SS_{0,s}(\gamma),\SS_{0,t}(\sigma)\rangle_{\text{HS}}\\
    &=K_{\text{sig}}^{\gamma, \sigma}(s,t).
\end{align*}
\end{proof}
As in \cite{SKlimit}, we can also prove that convergence occurs in $L^2$, and not just in expectation.
\begin{theorem}\label{thm: L2}
        Let $(\Omega,\FF,\P)$, $A_i^N$, $\gamma,\sigma$ and $Z_\gamma^N, Z_\sigma^N$ be as in \cref{thm: main_thm}, then
        \begin{equation}
            \frac{1}{N}\langle Z_\gamma^N(s),Z_\sigma^N(t)\rangle_\text{HS}\xrightarrow{L_2}K_{\text{sig}}^{\gamma,\sigma}(s,t).
        \end{equation}
\end{theorem}
\begin{proof}
    By expanding
    \[
    \Ex_{\xi_N}\Big[\Big(\frac{1}{N}\langle Z_\gamma^N(s),Z_\sigma^N(t)\rangle_{\text{HS}}-\langle \SS_{0,s}(\gamma),\SS_{0,t}(\sigma)\rangle_{\text{HS}}\Big)^2\Big]
    \]
    in the fashion of \cref{eq: sig_expansion}, we see that the objective is to prove
    \begin{equation}\label{eq: sig_expansion_quad}
    \begin{split}
        \lim_{N\to\infty} \Ex_{\xi_N}\bigg[&\sum_{|\bm{I}|,|\bm{K}|,|\bm{L}|,|\bm{K}|=0}^\infty \bigg(\frac{1}{N^{\frac{|\bm{I}^\star\bm{J}|}{2}+1}}\tr{A_{\bm{I}\star \bm{J}}^N} -\delta_{\bm{I}\bm{J}}\bigg)\cdot\\
        &\cdot\bigg(\frac{1}{N^{\frac{|\bm{K}^\star\bm{L}|}{2}+1}}\tr{A_{\bm{K}\star \bm{L}}^N} -\delta_{\bm{K}\bm{L}}\bigg)\hspace{1mm}\SS_{0,s}^{\bm{I}}(\gamma)\SS_{0,t}^{\bm{J}}(\sigma)\SS_{0,s}^{\bm{K}}(\gamma)\SS_{0,t}^{\bm{L}}(\sigma)\bigg]=0.
    \end{split}
    \end{equation}
    As in the proof of \cref{thm: main_thm}, we can show that exchanging the limit and expectation with the quadruple sum is justified. Multiplying out the brackets in \cref{eq: sig_expansion_quad} and taking limits of terms already dealt with in \cref{thm: main_thm} reduces the problem to showing
    \begin{equation}\label{eq: L2 expansion}
    \begin{split}
        \sum_{|\bm{I}|,|\bm{K}|,|\bm{L}|,|\bm{K}|=0}^\infty \bigg(\lim_{N\to\infty}&\frac{1}{N^{\frac{|\bm{I}^\star\bm{J}|+|\bm{K}^\star\bm{L}|}{2}+2}}\Ex_{\xi_N}\Big[\tr{A_{\bm{I}\star \bm{J}}^N}\tr{A_{\bm{K}\star \bm{L}}^N}\Big]\\
        &-\delta_{\bm{I}\bm{J}}\delta_{\bm{K}\bm{L}}\bigg)\hspace{1mm}\SS_{0,s}^{\bm{I}}(\gamma)\SS_{0,t}^{\bm{J}}(\sigma)\SS_{0,s}^{\bm{K}}(\gamma)\SS_{0,t}^{\bm{L}}(\sigma)=0.
    \end{split}
    \end{equation}
    Thus it remains to show that
    \[
    \lim_{N\to\infty}\frac{1}{N^{\frac{|\bm{I}^\star\bm{J}|+|\bm{K}^\star\bm{L}|}{2}+2}}\Ex_{\xi_N}\Big[\tr{A_{\bm{I}\star \bm{J}}^N}\tr{A_{\bm{K}\star \bm{L}}^N}\Big]=\delta_{\bm{I}\bm{J}}\delta_{\bm{K}\bm{L}}.
    \]
    Given closed words $\bm{w}=w_1\dots w_{n}w_1$ and $\bm{u}=u_1\dots u_m u_{1}$, we define the graphs $G_{\bm{w},\bm{u}}^{u}\coloneqq G_{\bm{w}}^{u}\cup G_{\bm{u}}^{u}$. Now we can write
    \begin{equation}
        \frac{1}{N^{\frac{|\bm{I}^\star\bm{J}|+|\bm{K}^\star\bm{L}|}{2}+2}}\Ex_{\xi_N}\Big[\tr{A_{\bm{I}\star \bm{J}}^N}\tr{A_{\bm{K}\star \bm{L}}^N}\Big]=\frac{1}{N^{\frac{|\bm{I}^\star\bm{J}|+|\bm{K}^\star\bm{L}|}{2}+2}}\sum_{\bm{w},\bm{u}}\mathscr{C}_{\bm{w},\bm{u}},
    \end{equation}
    where $\bm{w}=w_1 w_2\dots w_r w_{r+1} \dots w_n w_{1}$ and $\bm{u}=u_1 u_2\dots u_s u_{s+1} \dots u_m u_{1}$ and
    \[
    \begin{split}
    \mathscr{C}_{\bm{w},\bm{u}}\coloneqq \Ex_{\xi_N}\bigg[&{A_{i_{|\bm{I}|}}^{N^\ast}}(w_1,w_2) \dots {A_{i_1}^{N^\ast}}(w_r,w_{r+1}) {A_{j_1}^N}(w_{r+1},w_{r+2})\dots {A_{j_{|\bm{J}|}}^N}(w_{n},w_{1})\\
        &{A_{k_{|\bm{K}|}}^{N^\ast}}(u_1,u_2) \dots {A_{k_1}^{N^\ast}}(u_s,u_{s+1}) {A_{l_1}^N}(u_{s+1},u_{s+2})\dots {A_{l_{|\bm{L}|}}^N}(u_{m},u_{1})\bigg].
    \end{split}
    \]
    As in \cref{lem: correct_limit}, $N_{\bm{w},\bm{u}}^e$ must be at least $2$ for every $e\in E_{\bm{w},\bm{u}}^u$ in order for $\mathscr{C}_{\bm{w},\bm{u}}$ to not vanish. This implies that $|E_{\bm{w},\bm{u}}^u|\leq \tfrac{n+m}{2}$, and since $G_{\bm{w},\bm{u}}^{u}$ has at most two connected components, the restriction $\wt{\bm{w}\bm{u}}=|V_{\bm{w},\bm{u}}^u|\leq \tfrac{n+m}{2}+2$ is imposed. Moreover, $\wt{\bm{w}\bm{u}}=\tfrac{n+m}{2}+2$ if and only if $G_{\bm{w},\bm{u}}^{u}$ has two connected components, one with $\frac{n}{2}+1$ vertices and the other with $\frac{m}{2}+1$ vertices; in other words $G_{\bm{w}}^u$ and $G_{\bm{u}}^u$ are disjoint trees. In this case
    \[
\mathscr{C}_{\bm{w},\bm{u}}=\mathscr{C}_{\bm{w}}\big(A_{\bm{I}\star \bm{J}}^N\big)\mathscr{C}_{\bm{u}}\big(A_{\bm{K}\star \bm{L}}^N\big),
    \]
    since the words are disjoint. This reduces the problem to the one handled \cref{lem: correct_limit}. This product will be non-zero if and only if $\bm{I}=\bm{J}$ and $\bm{K}=\bm{L}$. Since there are $N(N-1)\dots \big(N-\tfrac{|\bm{I}^\star\bm{J}|+|\bm{K}^\star\bm{L}|}{2}-1\big)$ members of the equivalence class of $\bm{w}\bm{u}$, it holds that
    \small{
    \[
    \lim_{N\to\infty}\frac{\Ex_{\xi_N}\Big[\tr{A_{\bm{I}\star \bm{J}}^N}\tr{A_{\bm{K}\star \bm{L}}^N}\Big]}{N^{\frac{|\bm{I}^\star\bm{J}|+|\bm{K}^\star\bm{L}|}{2}+2}}=\lim_{N\to\infty}\frac{N(N-1)\dots \big(N-\tfrac{|\bm{I}^\star\bm{J}|+|\bm{K}^\star\bm{L}|}{2}-1\big)}{N^{\frac{|\bm{I}^\star\bm{J}|+|\bm{K}^\star\bm{L}|}{2}+2}}\delta_{\bm{I}\bm{J}}\delta_{\bm{K}\bm{L}}=\delta_{\bm{I}\bm{J}}\delta_{\bm{K}\bm{L}}.
    \]}
\end{proof}
\begin{remark}
    The assumptions of \cref{thm: main_thm,thm: L2} are satisfied in the case of independent entries of sub-Gaussian random variables under the condition
    \[
    \sup_{N\in\N}\sup_{1\leq i \leq d}\sup_{1\leq m,l\leq N} \norm{A_i^N(m,l)}_{\psi}<\infty,
    \]
    where $\norm{\cdot}_\psi$ denotes the sub-Gaussian norm. In this case, it is clear that \cref{eq: moment_condition} holds and it can be verified from \cite[Theorem 4.4.5]{Vershynin} that \cref{eq: sub_gaussian_condition} holds.
\end{remark}
\begin{remark}
    In \cite{SKlimit}, the authors in fact do not consider path developments directly (they also consider non-linear equations), rather randomised $\R^N$-valued CDEs. Consider the equations
    \begin{align*}\label{eq: CDEs}
        \dif X_\gamma^N(t) &= \frac{1}{\sqrt{N}}\sum_{i=1}^d A_i^NX_\gamma^N(t)\dif\gamma_t^i,\ X_\gamma^N(0)=\eta^N,\\
        \dif X_\sigma^N(t) &= \frac{1}{\sqrt{N}}\sum_{i=1}^d A_i^NX_\sigma^N(t)\dif\sigma_t^i,\ X_\sigma^N(0)=\eta^N,
    \end{align*}
    where the components of $\eta^N$ are independent with $(A_i^N)_{i=1}^d$ and $\eta^N$ also independent. As long as the components of $\eta^N$ have zero mean and unit variance, then it is straightforward to adjust the proof of \cref{thm: main_thm} to show that
    \begin{equation*}
        \lim_{N\to\infty} \Ex\Big[\frac{1}{N}\langle X_\gamma^N(s),X_\sigma^N(t)\rangle\Big]=K_{\text{sig}}^{\gamma,\sigma}(s,t).
    \end{equation*}
    If one further assumes that
    \[
    \sup_{N\in\N}\sup_{1\leq m\leq N}\Ex\Big[\big\vert{\eta^N(m)\big\vert}^4\Big]<\infty,
    \]
    then the convergence also happens in $L^2$ by adapting the proof of \cref{thm: L2}. We note that the corresponding \cref{eq: L2 expansion} will involve additional terms of the form
    \[
    \lim_{N\to\infty}\frac{1}{N^{\frac{|\bm{I}^\star\bm{J}|+|\bm{K}^\star\bm{L}|}{2}+2}}\Ex\Big[\tr{A_{\bm{I}\star \bm{J}}^NA_{\bm{K}\star \bm{L}}^N}\Big].
    \]
    Due to the scaling factor outside the equation, one can use the techniques from \cref{thm: main_thm} to show that the limit of such terms is zero.
\end{remark}
\section{Universal Limit of Unitary Developments and a Functional Equation}\label{sec: main_result}
Recently, it was proposed in \cite{PCFGAN} to use the unitary developments of a path as the basis for a discriminator in generative models. This can be seen as an extension of the characteristic function discriminator for vector valued data proposed by \cite{CFGAN,Ansari2019}. Consider the unitary Lie algebra and group of degree $N$ defined by
\begin{equation}
    \mathfrak{u}_N=\{A\in\MM_N(\C) : A^\star = -A\},\quad U(N;\C)=\{A\in\MM_N(\C):A^\star A = I_N\}.
\end{equation}
Note that $\mathfrak{u}_N=i\mathfrak{h}_N$, where $\mathfrak{h}_N$ is the set of $N\times N$ Hermitian matrices. For any path $\gamma\in\XX$ and $M\in\text{Hom}(V, \mathfrak{u}_n)$ the unitary feature of $\gamma$ under $M$ is the time-$T$ solution to the controlled differential equation
\begin{equation}\label{eq: pcf eq}
    \dif Z_t = Z_t\cdot M(\dif \gamma_t),\quad Z_0 = I_N.
\end{equation}
The solution $\UU(\gamma; M) \coloneqq Z_T$ is a special case of the Cartan development of $\gamma$ taking values in the Lie group $U(N;\C)$. Let $\mu$ be a Borel probability measure on $\XX$ and $\gamma\sim\mu$, then the path characteristic function (PCF) of $\mu$ evaluated at $M$ is defined as
\begin{equation}
    \Phi_{\mu}(M)\coloneqq\Ex_{\gamma\sim\mu}\big[\UU(\gamma; M)\big].
\end{equation}
\begin{definition}[PCFD \cite{PCFGAN}]
    Let $\mu,\nu$ be probability measures on $\XX$ with $\gamma\sim\mu$ and $\sigma\sim\nu$, and let $\xi$ be a probability measure on $\text{Hom}(V,\mathfrak{u}_N)$ with $M\sim \xi$. The path characteristic function distance between $\mu$ and $\nu$ with respect to $\xi$ is defined as
    \begin{equation}\label{eq: pcfd}
    d_{\xi}^2(\mu,\nu)\coloneqq\Ex_{M\sim\xi}\big[\normhs{\Phi_\mu(M)-\Phi_\nu(M)}^2\big].
    \end{equation}
\end{definition}
For fundamental results on PCFD that allow for its use as a loss function in applications as well as certain theoretical properties, we refer the reader to \cite{PCFGAN}. In \cite[Proposition B.10]{PCFGAN} it was shown that $d_\xi(\mu,\nu)$ is a maximum mean discrepancy (MMD) distance on the space $\PP(\XX)$, the set of Borel probability measures on $\XX$.
\begin{definition}
    Given a kernel function $\kappa:\XX\times\XX\to\R$, then the MMD distance associated to $\kappa$ is $\mathrm{MMD}_\kappa:\PP(\XX)\times \PP(\XX)\to \R^{+}$ is defined by
    \begin{equation}
        \mathrm{MMD}_\kappa^2(\mu,\nu)\coloneqq \Ex_{\gamma,\gamma^\prime\stackrel{\mathrm{iid}}{\sim}\mu}\big[\kappa(\gamma, \gamma^\prime)\big]+\Ex_{\sigma,\sigma^\prime\stackrel{\mathrm{iid}}{\sim}\nu}\big[\kappa(\sigma, \sigma^\prime)\big]-2\Ex_{\gamma\sim\mu,\sigma\sim\nu}\big[\kappa(\gamma,\sigma)\big].
    \end{equation}
\end{definition}
\begin{proposition}[PCFD as an MMD: Proposition B.10 \cite{PCFGAN}] 
    Let $\xi\in\PP\big(\mathfrak{u}_N\big)$, $\mu,\nu\in\PP(\XX)$, then $d_\xi(\mu,\nu)=\mathrm{MMD}_\kappa(\mu,\nu)$, where
    \begin{equation}\label{eq: pcf_kernel}
        \kappa(\gamma,\sigma)\coloneqq \Ex_{M\sim\xi}\big[\UU(\gamma\star\overleftarrow{\sigma};M)\big].
    \end{equation}
\end{proposition}
One of the difficulties in the use of PCFD for the learning of a distribution on high-dimensional or complex paths is the computational cost of  computing $\UU(\gamma; M)$, $M\in\text{Hom}(V, \mathfrak{u}_n)$ for large $N$, especially for a fully supported measure on $\text{Hom}(V, \mathfrak{u}_n)$. The authors of \cite{PCFGAN} propose two solutions: firstly they employ a reciprocal learning approach where they embed high dimensional paths into a lower dimensional space, where the embedding should act as a type of pseudo-inverse of the original generator; secondly, they adversarially learn a finitely supported empirical measure on $\text{Hom}(V, \mathfrak{u}_n)$. We propose an alternative approach, where we take the dimension $N\to\infty$ for a class measures supported on the whole space. The limiting kernel then solves a functional equation driven by $\gamma\star\overleftarrow{\sigma}$. Before we state and prove our result, we introduce  non-crossing pair partitions and a new associated notion of generations.
\begin{remark}
    For any fixed $N<\infty$, the authors of \cite{PCFGAN} remark that computing \cref{eq: pcfd} directly is linear in sample size, whereas computing it as an MMD through \cref{eq: pcf_kernel} is quadratic in sample size.
\end{remark}
\subsection{Non-Crossing Partitions, Dyck Words and Trees}
From this point on, we will consider the alphabet $\AA_n$ to be the ordered set $\{1,\dots,n\}$. All the definitions may, however, be carried over to any ordered set of $n$ letters. An element of a partition of $\AA_n$ will be called a ``part''.
    \begin{definition}
        A partition $\pi$ of $\AA_n$ is called crossing if there exist $1\leq a<b<c<d\leq n$ for which $a,c$ belongs to one part of $\pi$ and $b,d$ to another part. A partition that is not crossing is called non-crossing. 
    \end{definition}
Of particular importance to us is the class of non-crossing \textit{pair partitions}.
    \begin{definition}
        A pair partition of $\AA_n$ is one where each part is a two element set. We denote by $\text{NC}_2(n)$ the set of non-crossing pair partitions of $\AA_n$.
    \end{definition}
For a non-crossing pair partition $\pi$ and $\{i,j\}\in\pi$, we define $\pi_{\{i,j\}}=\pi_i=\pi_j\coloneqq \min\{i,j\}$. A convenient method of encoding non-crossing pair partitions is through Dyck words. We recall their definition from \cite{TAO_RMT}[Section 2].
\begin{definition}[Dyck Words]
    A Dyck word is a sequence of left and right parentheses, ``$($'' and ``$)$'', such that at every stage, when read from left-to-right, there are at least as many left parentheses as right parentheses and exactly the same number of each at the end of the word. We denote by $\DD_n$ the set of Dyck words of length $n$.
\end{definition}
  Note that $\DD_n=\varnothing$ if $n$ is odd. There is a well-established bijection $\text{NC}_2(n)\to \DD_n$ where one reads the letters $1\dots n$ from left to right and for each letter, if it is the first element in a pair, then one writes a left parenthesis ``$($'', and if it is the second letter in a pair one writes a right parenthesis ``$)$''. For a proof that this map indeed defines a bijection, see \cite{TAO_RMT}[Lemma 2.3.15].

It is also helpful to note the bijection between non-crossing pair partitions $\text{NC}_2(n)$ and the set of non-planar rooted trees $\TT_n$ with vertices $\{1,\dots,\frac{n}{2} + 1\}$ and $\frac{n}{2}$ edges, where a left-to-right traversal of the tree visits the vertices in lexicographical order. Recall that a left-to-right traversal starts at the root node and proceeds with the condition that for any two subtrees of any given node, all vertices of the subtree left subtree must be visited before the vertices of the right subtree.

Once again, we refer the reader to \cite{TAO_RMT}[Section 2] and \cite{RMintro}[Chapter 2] for full details on these bijections. An example of these bijections for a non-crossing pair-partition of $\AA_8$ is illustrated below.
\[
\big\{\{1,6\},\{2,3\},\{4,5\},\{7,8\}\big\}\equiv \tikz[baseline = -0.1ex]{
			\draw[fill] (0,0) circle [radius=0.04];
			\draw[fill] (0.3,0) circle [radius=0.04];
			\draw[fill] (0.6,0) circle [radius=0.04];
			\draw[fill] (0.9,0) circle [radius=0.04];
			\draw[fill] (1.2,0) circle [radius=0.04];
			\draw[fill] (1.5,0) circle [radius=0.04];
			\draw[fill] (1.8,0) circle [radius=0.04];
                \draw[fill] (2.1,0) circle [radius=0.04];
			\draw (0,0) .. controls (0,0.4) and (1.5,0.4) .. (1.5,0);
			\draw (0.3,0) .. controls (0.3,0.25) and (0.6,0.25) .. (0.6,0);
                \draw (0.9,0) .. controls (0.9,0.25) and (1.2,0.25) .. (1.2,0);
                \draw (1.8,0) .. controls (1.8,0.25) and (2.1,0.25) .. (2.1,0);
                \node[below] at (0,0) {$\scriptstyle{1}$};
                \node[below] at (0.3,0) {$\scriptstyle{2}$};
                \node[below] at (0.6,0) {$\scriptstyle{3}$};
                \node[below] at (0.9,0) {$\scriptstyle{4}$};
                \node[below] at (1.2,0) {$\scriptstyle{5}$};
                \node[below] at (1.5,0) {$\scriptstyle{6}$};
                \node[below] at (1.8,0) {$\scriptstyle{7}$};
                \node[below] at (2.1,0) {$\scriptstyle{8}$};
		}\mapsto(()())()\mapsto \TreeOneTwoTwo{1}{2}{3}{4}{5}.
\]
\subsubsection{Generations}\label{sec: generations}
We are interested in generating, in a unique sequence of steps, any non-crossing pair partition $\pi$ from the partition $\{\{1,2\}\}$ of $\{1,2\}$, or equivalently any Dyck word from $()$, or any tree from $\SmallTreeOneOne{1}{2}$. We define the generation of a pair of parentheses in a Dyck word recursively by:
\begin{enumerate}[label = \arabic*)]
    \item The pair whose right parenthesis, ``$)$'', is the final parenthesis in the Dyck word has generation $1$.
    \item A pair of parentheses is of generation $k+1$ if the parenthesis in the Dyck word following the right-parenthesis of the pair is part of a pair of generation $k$.
\end{enumerate}
Alternatively we can define the generation of an edge of a tree via:    
\begin{enumerate}[label=\arabic*)]
    \item The right-most edge attached to the root has generation $1$.
    \item An edge is of generation $k+1$ if it is the right-most edge attached above an edge of generation $k$, or it is the right-most edge attached to the left of an edge of generation $k$.
\end{enumerate}
\begin{example}
Consider the partition $\pi\in\NC{12}$:
\begin{equation*}
    \pi =
    \tikz[baseline = -0.1ex]{
			\draw[fill] (0,0) circle [radius=0.04];
			\draw[fill] (0.32,0) circle [radius=0.04];
			\draw[fill] (0.64,0) circle [radius=0.04];
			\draw[fill] (0.96,0) circle [radius=0.04];
			\draw[fill] (1.28,0) circle [radius=0.04];
			\draw[fill] (1.6,0) circle [radius=0.04];
			\draw[fill] (1.92,0) circle [radius=0.04];
                \draw[fill] (2.24,0) circle [radius=0.04];
                \draw[fill] (2.56,0) circle [radius=0.04];
                \draw[fill] (2.88,0) circle [radius=0.04];
                \draw[fill] (3.2,0) circle [radius=0.04];
                \draw[fill] (3.52,0) circle [radius=0.04];
			\draw (0,0) .. controls (0,0.25) and (0.32,0.25) .. (0.32,0);
			\draw (0.64,0) .. controls (0.64,0.25) and (0.96,0.25) .. (0.96,0);
                \draw (1.28,0) .. controls (1.28,0.6) and (3.52,0.6) .. (3.52,0);
                \draw (1.6,0) .. controls (1.6,0.25) and (1.92,0.25) .. (1.92,0);
                \draw (2.24,0) .. controls (2.24,0.4) and (3.2,0.4) .. (3.2,0);
                \draw (2.56,0) .. controls (2.56,0.25) and (2.88,0.25) .. (2.88,0);
                \node[below] at (0,0) {$\scriptstyle{1}$};
                \node[below] at (0.32,0) {$\scriptstyle{2}$};
                \node[below] at (0.64,0) {$\scriptstyle{3}$};
                \node[below] at (0.96,0) {$\scriptstyle{4}$};
                \node[below] at (1.28,0) {$\scriptstyle{5}$};
                \node[below] at (1.6,0) {$\scriptstyle{6}$};
                \node[below] at (1.92,0) {$\scriptstyle{7}$};
                \node[below] at (2.24,0) {$\scriptstyle{8}$};
                \node[below] at (2.56,0) {$\scriptstyle{9}$};
                \node[below] at (2.86,0) {$\scriptstyle{10}$};
                \node[below] at (3.2,0) {$\scriptstyle{11}$};
                \node[below] at (3.54,0) {$\scriptstyle{12}$};
		}.
\end{equation*}
The Dyck word and tree corresponding to $\pi$ are given by
\begin{equation}
    {\color{imperialGreen}(}{\color{imperialGreen})}{\color{imperialBlue}(}{\color{imperialBlue})}{\color{imperialRed}(}{\color{imperialGreen}(}{\color{imperialGreen})}{\color{imperialBlue}(}{\color{imperialGreen}(}{\color{imperialGreen})}{\color{imperialBlue})}{\color{imperialRed})},\quad\text{and}\quad \TreeOneThreeTwoOne{1}{2}{3}{4}{5}{6}{7}.
\end{equation}
Here, generation one parentheses or edges have been drawn in red, generation two has been drawn in blue and generation three parentheses or edges are green.
\end{example}
The generation of a Dyck word is the maximum of the generations of its constituent pairs of parentheses. Likewise, the generation of a tree is the maximum of the generations of its edges. The generation of a pair in a non-crossing pair partition is the generation of the corresponding pair of parentheses its Dyck word. Similarly, the generation of the partition is defined as the generation of the corresponding Dyck word. For each $k\in\N$, define $G_k$ to be the set of non-crossing pair partitions (Dyck words or trees) of generation $k$, and for any non-crossing pair partition $\pi$, we let $g(\pi)$ be the set of parts (pairs or edges) whose generation is equal to the generation of $\pi$. In the case $k=0$, we also define $G_0\coloneqq\{\varnothing\}$.
\begin{remark}
    It can be shown that the notion of generation and its properties are preserved by the aforementioned bijection between Dyck words and trees.
\end{remark}
For any Dyck word $d$, we let $g(d)$ be the set of pairs of parentheses with maximal generation. Additionally, we define $G(d)$ to be the set of Dyck words obtained by inserting a pair ``$()$'' immediately to the left of at least one parenthesis that is part of a pair in $g(d)$.
\begin{lemma}\label{lem: generation_properties}
    The following facts about generations hold.
    \begin{enumerate}[label=\arabic*)]
        \item For any non-crossing pair partition $\pi$, any $\{i,j\}\in g(\pi)$ is of the form $\{i,i+1\}$.
        \item Let $d_1\neq d_2$ be Dyck words of generation $k$, then $G(d_1)\cap G(d_2)=\varnothing$.
        \item We have that $G_{n+1}=\bigsqcup_{\pi\in G_n}G(\pi)$.
    \end{enumerate}
\end{lemma}
\begin{proof}
If any pair of parentheses in a Dyck word is not consecutive, then they cannot have the highest generation: if they are not consecutive, they must enclose a pair that then has a higher generation. Noting that consecutively paired parentheses in a Dyck word correspond to consecutively paired letters in a non-crossing pair partition proves the first point. Since insertion only happens immediately to the \textbf{left} of a parenthesis of generation $k$, the generation of an inserted pair must be of generation $k+1$ and moreover any pair of generation $k+1$ must have been inserted. Hence, by deleting the pairs of generation $k+1$ in $d\in G(d_1)\cap G(d_2)$ we should obtain both $d_1$ and $d_2$, from which we deduce $d_1=d_2$. The third point is an immediate corollary of the second.
\end{proof}
\subsection{A Limiting Functional Equation}
Recall that $\mathfrak{u}_N=i\mathfrak{h}_N$, so that, as in \cref{sec: Sk}, \cref{eq: pcf eq} can be written in the form
\begin{equation*}
    \dif Z_t=i\sum_{i=1}^dZ_tA_i\dif\gamma_t^i,\quad Z_0=I_N,
\end{equation*}
for a collection of $d$ matrices in $\mathfrak{h}_N$. We are now ready to state the main result of this section.
\begin{theorem}\label{thm: fpcf_eq}
     Let $(\Omega, \FF, \P)$ be a probability space supporting for each $N\in\N$ and $1\leq i \leq d$, a Hermitian random matrix $A_i^N:\Omega \to \mathfrak{u}_N$. Assume that for all $k\in\N$,
    \begin{equation}
        \sup_{N\in\N}\sup_{1\leq i \leq d}\sup_{1\leq m,l\leq N}\Ex\big[\big\vert{A_i^N(m,l)}\big\vert^k\big]\leq c_A(k)< \infty,
    \end{equation}
    and that $\Ex\big[A_i^N(m,l)\big]=0$ and $\Ex\big[\big\vert A_i^N(m,l)\big\vert^2\big]=1$. Suppose also that there exists some $\kappa>0$ such that for every $p\in\N$
    \begin{equation}
        \sup_{N\in\N}\sup_{1\leq i \leq d}\Ex\bigg[\frac{1}{N^{\frac{p}{2}}}\normopp[\big]{A_i^N}\bigg]\leq \kappa^p\Gamma(\tfrac{p}{2}+1),
    \end{equation}
    where $\Gamma(\cdot)$ is the Gamma function. For each $N\in\N$, we assume that
    \begin{equation}
        \bigvee_{i=1}^d\sigma\big(A_i^N(m,l\big)\quad \text{for } 1\leq m\leq l\leq N,
    \end{equation}
    are independent sigma algebras and that
    \begin{equation}
    \Ex\big[A_i^N(m,l)A_j^N(m,l)\big]=0,\quad\text{for }i\neq j.
    \end{equation}
   For each $N$, let $\xi_N$ be the law of the collection in $\PP(\mathfrak{h}_N^d)$. Let $\gamma\in\XX$ and let $Z^N_\gamma(s,t)$ be the $U(N;\C)$-valued solution to
    \begin{equation}\label{eq: hermit}
        \dif Z_\gamma^N(s,t)=\frac{i}{\sqrt{N}}\sum_{j=1}^dZ_{\gamma}^N(s, t)A_j^N \dif\gamma_t^j,\quad Z_\gamma^N(s,s)=I_N.
    \end{equation}
    Then, for $0\leq s\leq t\leq T$, $K_\gamma(s,t)\coloneqq \lim_{N\to\infty} \frac{1}{N}\Ex_{\xi_N}\Big[\tr{Z^N_\gamma(s,t)}\Big]$ exists and solves the functional equation
    \begin{equation}\label{eq: functional_eq}
        K_\gamma(s,t)=1-\int_s^t\int_s^r K_\gamma(s,u)K_\gamma(u,r) \langle \dif \gamma_u,\dif \gamma_r\rangle,\quad K_\gamma(s,s)=K_\gamma(t,t)=1.
    \end{equation}
\end{theorem}
\begin{proof}
   As in the proof of \cref{thm: main_thm}, we may expand the solution $Z_\gamma(s,t)$ in terms of the signature of $\gamma$. Indeed, we may write
\begin{equation*}
    \frac{1}{N}\tr{Z_\gamma(s,t)}=\sum_{|\bm{I}|=0}^\infty \frac{i^{|\bm{I}|}}{N^{\frac{|\bm{I}|}{2}+1}}\tr{{A_{\bm{I}}^N}}\SS_{s,t}^{\bm{I}}(\gamma).
\end{equation*}
The same techniques in the proof of \cref{thm: main_thm} show that the limits, expectations and sums may be interchanged. Meanwhile, \cite[Theorem 5.4.2]{RMintro} gives
\begin{equation*}
    \lim_{N\to\infty}\frac{1}{N^{\frac{|\bm{I}|}{2}+1}}\Ex_{\xi_N}\Big[\tr{A_{\bm{I}}^N}\Big]=\varphi(\bm{I}),
\end{equation*}
where $\varphi(\cdot)$ denotes the moment sequence of a collection of $d$-free semi-circular random variables. Hence
\begin{equation}\label{eq: expansion}
    K_\gamma(s,t)= \lim_{N\to\infty} \frac{1}{N}\Ex_{\xi^N}\Big[\tr{Z^N_\gamma(s,t)}\Big]=\sum_{|\bm{I}|=0}^\infty i^{|\bm{I}|}\varphi(\bm{I})\SS_{s,t}^{\bm{I}}(\gamma),
\end{equation}
We recall from \cite[Lemma 5.4.7, Remark 5.4.8]{RMintro} that the moment sequence $\varphi(\cdot)$ of $d$-free semi-circular random variables is uniquely characterised by the Schwinger-Dyson equations
\begin{align}
    \varphi(\bm{I}\bm{J})&=\varphi(\bm{J}\bm{I}),\\
    \varphi(\bm{I}j)&=\sum_{\bm{I}=\bm{K}j\bm{L}}\varphi(\bm{K})\varphi(\bm{L}).
\end{align}
And so we may write
\begin{align*}
    K_\gamma(s,t)&=\sum_{|\bm{I}|=0}^\infty i^{|\bm{I}|}\varphi(\bm{I})\SS_{s,t}^{\bm{I}}(\gamma)\\
    &=1+\sum_{j=1}^d\sum_{|\bm{I}|=0}^\infty i^{|\bm{I}|+1}\varphi(\bm{I}j)\SS_{s,t}^{\bm{I}j}(\gamma)\\
    &=1+\sum_{j=1}^d\sum_{|\bm{K}|,|\bm{L}|=0}^\infty i^{|\bm{K}|+|\bm{L}| + 2}\varphi(\bm{K})\varphi(\bm{L})\SS_{s,t}^{\bm{K}j\bm{L}j}(\gamma).
\end{align*}
Letting $\bm{L}=l_1\dots l_k$, then
\begin{align*}
    \SS_{s,t}^{\bm{K}j\bm{L}j}(\gamma) &= \int\limits_{s<u<u_1<\dots<u_{k}<r<t}\SS_{s,u}^{\bm{K}}(\gamma)\dif\gamma_{u}^j\dif \gamma_{u_1}^{l_1}\dots\dif \gamma_{u_k}^{l_k}\dif \gamma_{r}^j\\
    &=\int_{s<u<r<t}\SS_{s,u}^{\bm{K}}(\gamma) \dif\gamma_{u}^j\int\limits_{u<u_1<\dots<u_k<r} \dif \gamma_{u_1}^{l_1}\dots\dif \gamma_{u_k}^{l_k}\dif \gamma_{r}^j\\
    &=\int_{s<u<r<t} \SS_{s,u}^{\bm{K}}(\gamma)\SS_{u,r}^{\bm{L}}(\gamma)\dif\gamma_u^j\dif\gamma_r^j.
\end{align*}
Substituting this expression into the expansion of $K_\gamma(s,t)$ gives
\begin{align*}
    K_\gamma(s,t)&= 1-\sum_{j=1}^d\sum_{|\bm{K}|,|\bm{L}|=0}^\infty i^{|\bm{K}|+|\bm{L}|}\varphi(\bm{K})\varphi(\bm{L})\int_{s<u<r<t} \SS_{s,u}^{\bm{K}}(\gamma)\SS_{u,r}^{\bm{L}}(\gamma)\dif\gamma_u^j\dif\gamma_r^j\\
    &=1-\sum_{j=1}^d\int_{s<u<r<t} \left(\sum_{|\bm{K}|=0}^\infty i^{|\bm{K}|}\varphi(\bm{K})\SS_{s,u}^{\bm{K}}(\gamma)\right)\left(\sum_{|\bm{L}|=0}^\infty i^{|\bm{L}|}\varphi( \bm{K})\SS_{u,r}^{\bm{L}}(\gamma)\right)\dif\gamma_u^j\dif\gamma_r^j\\
    &=1-\int_s^t\int_s^r K_\gamma(s,u)K_\gamma(u,r) \langle \dif \gamma_u,\dif \gamma_r\rangle.
\end{align*}
\end{proof}
\begin{lemma}\label{lem: func_unique}
For any $\gamma\in\XX$, the functional equation
\begin{equation}\label{eq: func_eq}
        K(s,t)=1-\int_s^t\int_s^r K(s,u)K(u,r) \langle \dif \gamma_u,\dif \gamma_r\rangle,\quad K(s,s)=K(t,t)=1,
    \end{equation}
    has a unique solution for $0\leq s\leq t\leq T$.
\end{lemma}
Due to the quadratic nature of the equation, we adopt the approach of iteratively expanding the equation and proving convergence to the series in \cref{eq: expansion}.
\begin{proof}
    Existence of a solution is settled by \cref{thm: fpcf_eq}. To prove uniqueness, we show that any solution $K$ must agree with the solution $K_\gamma$ described in \cref{thm: fpcf_eq}. Fix a solution $K$ to \cref{eq: func_eq}, then it must be uniformly bounded on the simplex by the structure of \cref{eq: func_eq}. Define $\kappa=\max\{1,\norm{K}_{\infty}\}$. We now introduce some notation. For any $\pi \in \NC{2k}$, recall that $\pi_i$ is the minimum value of the pair to which $i$ belongs. We then define following:
    \begin{enumerate}[label=\arabic*)]
        \item The time-ordered simplex $\Delta^{\pi}[s,t]\coloneqq\{s=t_0\leq t_1\leq\dots\leq t_{2k}\leq t\}$.
        \item The ``differential'' $\dif\gamma^{\pi}\coloneqq (-1)^k\sum_{j_{\pi_1},\dots,j_{\pi_{2k}}=1}^d\dif\gamma_{t_1}^{j_{\pi_1}}\dif\gamma_{t_2}^{j_{\pi_2}}\dots \dif\gamma_{t_{2k}}^{j_{\pi_{2k}}}$. Note that the sum is over $k$ free indices since $j_{\pi_i}=j_{\pi_j}$ if $\{i,j\}\in\pi$.
        \item The integrand 
        \begin{equation*}
            K^\pi\coloneqq\prod_{\{i,i+1\}\in g(\pi)}K(t_{i-1},t_{i})K(t_{i},t_{i+1}).
        \end{equation*}
    \end{enumerate}
    Hence, for any even $k$ and $\pi\in\NC{k}$, the integral
    \begin{equation*}
        I(\pi)\coloneqq\int_{\Delta^{\pi}[s,t]}K^\pi \dif \gamma^\pi
    \end{equation*}
    is well defined. We also define
    \begin{equation*}
        \SS_{s,t}^\pi(\gamma)\coloneqq \int_{\Delta^{\pi}[s,t]}\dif \gamma^\pi.
    \end{equation*}
    As usual, we set $\SS^\varnothing_{s,t}(\gamma)\equiv1$. We recall that $G_k$ is the set of non-crossing pair-partitions of generation $k$. Using this language, $K$ satisfies the equation
    \begin{equation}
        K(s,t)=\sum_{\pi\in G_0}\SS^\pi_{s,t}(\gamma)+\sum_{\pi\in G_1} I(\pi),
    \end{equation}
    since $G_0=\{\varnothing\}$ and $G_1 = \{\{1,2\}\}$. We now claim that
    \begin{equation}\label{eq: inductive_functional_expansion}
        K(s,t)=\sum_{\pi\in \bigcup_{i=0}^{n-1} G_i}\SS^\pi_{s,t}(\gamma)+\sum_{\pi\in G_n} I(\pi).
    \end{equation}
    To see this, fix $\pi\in G_n$ and for each $\{i,i+1\}\in g(\pi)$, the terms in the product $K_\gamma^\pi$ associated with $\{i,i+1\}$ are given by
    \[
    \begin{split}
         K(t_{i-1},t_{i})K(t_{i},t_{i+1})=&\left(1-\int_{t_{i-1}}^{t_{i}}\int_{t_{i-1}}^{r_1}K(t_{i-1},u_1)K(u_1,r_1)\sum_{l_1=1}^d\dif\gamma_{u_1}^{l_1}\dif\gamma_{r_1}^{l_1}\right)\\
         &\left(1-\int_{t_{i}}^{t_{i+1}}\int_{t_{i}}^{r_2}K(t_{i},u_2)K(u_2,r_2)\sum_{l_2=1}^d\dif\gamma_{u_2}^{l_2}\dif\gamma_{r_1}^{l_2}\right).
    \end{split}
    \]
    Now, from each bracket, we may either select the constant $1$ or the double integral. We see that the four possible selections correspond to replacing the pair ${\color{imperialRed}()}$, corresponding to $\{i,i+1\}$ in the Dyck word associated with $\pi$, with ${\color{imperialRed}()}$, ${\color{imperialBlue}()}{\color{imperialRed}()}$, ${\color{imperialRed} (}{\color{imperialGreen}()}{\color{imperialRed})}$ or ${\color{imperialBlue}()}{\color{imperialRed} (}{\color{imperialGreen}()}{\color{imperialRed})}$. By applying these selections to each pair of brackets in $K^\pi$, we see that
    \[
    I(\pi)=\SS^\pi_{s,t}(\gamma)+\sum_{\pi^\prime\in G(\pi)}I(\pi^\prime).
    \]
    Now
    \begin{align*}
        K(s,t)&=\sum_{\pi\in \bigcup_{i=0}^{n} G_i}\SS^\pi_{s,t}(\gamma)+\sum_{\pi\in G_{n+1}} I(\pi)\\
        &=\sum_{|\pi|\leq 2n}\SS^\pi_{s,t}(\gamma)+\sum_{\substack{\pi\in \bigcup_{i=0}^{n} G_i\\|\pi|>2n}}\SS^\pi_{s,t}(\gamma)+\sum_{\pi\in G_{n+1}} I(\pi)\\
        &=\sum_{|\bm{I}|=0}^{2n}(-1)^{\frac{|\bm{I}|}{2}}\SS_{s,t}^{\bm{I}}(\gamma)\sum_{\pi\in \NC{|\bm{I}|}}\prod_{\{i,j\}\in\pi}\id_{\bm{I}_i=\bm{I}_j}+\sum_{\substack{\pi\in \bigcup_{i=0}^{n} G_i\\|\pi|>2n}}\SS^\pi_{s,t}(\gamma)+\sum_{\pi\in G_{n+1}} I(\pi)\\
        &=\sum_{|\bm{I}|=0}^{2n}(-1)^{\frac{|\bm{I}|}{2}}\varphi(\bm{I})\SS^{\bm{I}}_{s,t}(\gamma)+\sum_{\substack{\pi\in \bigcup_{i=0}^{n} G_i\\|\pi|>2n}}\SS^\pi_{s,t}(\gamma)+\sum_{\pi\in G_{n+1}} I(\pi),
        \end{align*}
    where the final equality follows from another characterisation of the moments of $d$-free semicircular random variables; see, for example, \cite{RMintro}[Chapter 5, pp.380-381]. We observe that the first term converges to $K_\gamma(s,t)$, so it remains to show that the second two sums converge to $0$ as $n\to\infty$. Here, we note that, by the uniqueness of the construction of Dyck words, for a word $\bm{J}$, with $|\bm{J}|>2n$, the differential $\dif^{\bm{J}}$ appears at most $\varphi(\bm{J})$ times in the second two sums. And so we have that
    \begin{align*}
        \left\vert\sum_{\substack{\pi\in \bigcup_{i=0}^{n} G_i\\|\pi|>2n}}\SS^\pi_{s,t}(\gamma)+\sum_{\pi\in G_{n+1}} I(\pi)\right\vert &\leq \sum_{|\bm{J}|>2n}\varphi(\bm{J})\kappa^{|\bm{J}|}\left\vert \SS^{\bm{J}}_{s,t}(\gamma)\right\vert\\
        &\leq \sum_{|\bm{J}|>2n}C_{\frac{|\bm{J}|}{2}}\kappa^{|\bm{J}|}\left\vert \SS^{\bm{J}}_{s,t}(\gamma)\right\vert\\
        &\to 0,
    \end{align*}
    as $n\to \infty$, where $C_k$ is the $k$\textsuperscript{th} Catalan number. The convergence is due to the factorial decay of the signature and the well known fact that $C_k=\OO\big(4^k\big)$. And so the solution is unique.
\end{proof}
The results of this section show that the limit of $d_{\xi_N}^2(\mu,\nu)$ exists under mild conditions on the sequence $(\xi_N)_{N\in\N}$. The limit is an MMD, with kernel given by
\begin{equation}
    K_{\text{SD}}(\gamma,\sigma):=K_y(0,T),\quad\text{with }y\coloneqq \gamma\ast\overleftarrow{\sigma}.
\end{equation}
We leave it to future work to investigate the discriminative properties of $K_{\text{SD}}$ and how it compares with the adversarially learnt $d_{\xi}$ proposed by \cite{PCFGAN}. We note that replacing $d_{\xi}$ with the MMD based on $K_{\text{SD}}$ is akin to the non-adversarial training of neural-SDEs proposed by \cite{issa2023nonadversarial} using the the MMD based on $K_{\text{sig}}$.
\section{Numerical Schemes}\label{sec: numerical}
In order for $K_{\text{SD}}$ to have practical application, we require some method of approximating the solution to \cref{eq: functional_eq} without the need for computing the signature $\SS(\gamma)$. We propose explicit and implicit numerical schemes based on piecewise constant approximations to $\gamma$ and solving the corresponding functional equation, interpreting the iterated integrals as left-or right-point Riemann-Stieltjes integrals. 
\subsection{Explicit Scheme}\label{sec: explicit}
We propose a straightforward explicit numerical scheme to solve \cref{eq: functional_eq}. We replace $\gamma$ with a piecewise constant approximation $\gamma^\pi$ and solve the corresponding integral equation for $K_{\gamma^\pi}$. The resulting solution is also given by the moments of free semicircular random variables contracted against the iterated-sums signature (or discrete-time signature) \cite{KO, ISS} of $\gamma^\pi$. This latter expression will allow us to determine a convergence rate for the scheme.

Assume that $\gamma\in \XX$ and that $\pi=\{0=t_1<\dots<t_n<t_{n+1}=T\}$ is a partition of $[0,T]$. Define $\gamma^\pi$ by
\[
\gamma_t^\pi=\sum_{i=1}^n\id_{[t_i,t_{i+1})}\gamma_{t_i}.
\]
Note that $\gamma^{\pi}$ is c\'adl\'ag on $[0,T]$. Consider the functional equation driven by $\gamma^\pi$
\begin{equation*}
    K_{\gamma^\pi}(s,t)=1-\int_s^t\int_s^rK_{\gamma^\pi}(s,u)K_{\gamma^\pi}(u,r)\langle \dif\gamma_u^\pi,\dif\gamma_r^\pi\rangle,\quad K(s,s)=1,
\end{equation*}
where the double integral is interpreted as iterated left-point Riemann-Stieltjes integrals, see \cite{FZ_jumps} for details. We introduce the jumps of $\gamma^\pi$ as
\[
\Delta_{t}\gamma^\pi = \lim_{s\to t}\gamma^\pi_t-\gamma^\pi_s,
\]
so that $\Delta_{t_{i+1}}\gamma^\pi=\gamma^\pi_{t_{i+1}}-\gamma^\pi_{t_i}$. For $t\in [0,T]$, let $\tau(t)=\max\{i:t_i\leq t\}$. We may then expand $K_{\gamma^\pi}$ as
\begin{align*}
    K_{\gamma^\pi}(s,t) &= 1 - \sum_{j=\tau(s)}^{\tau(t)-1}\sum_{i=\tau(s)}^{j-1} K_{\gamma^\pi}(t_{\tau(s)},t_i)K_{\gamma^\pi}(t_i,t_j)\langle \Delta_{t_{i+1}}\gamma^\pi,\Delta_{t_{j+1}}\gamma^\pi\rangle\\
    &=1 - K_{\gamma^\pi}(s,t_{\tau(t)-1})-\sum_{i=\tau(s)}^{\tau(t)-1} K_{\gamma^\pi}(t_{\tau(s)},t_i)K_{\gamma^\pi}(t_i,t_{\tau(t)-1})\langle \Delta_{t_{i+1}}\gamma^\pi,\Delta_{t_{\tau(t)}}\gamma^\pi\rangle.
\end{align*}
Since $K_{\gamma^\pi}$ is piecewise constant on the two-simplex, it is enough to evaluate $K_{\gamma^\pi}$ at times in
\[
(t_i,t_{j})\in\Pi:=\pi^2\cap\{(s,t):0\leq s \leq t \leq T\} 
\]
in the solution grid $\Pi$. Doing so yields the recursive set of equations
\begin{equation}\label{eq: explicit_scheme}
    K_{\gamma^\pi}(t_i,t_{i+j})=1 - K_{\gamma^\pi}(t_i,t_{i+j})-\sum_{l=i}^{i+j-1} K_{\gamma^\pi}(t_{i},t_l)K_{\gamma^\pi}(t_l,t_{i+j-1})\langle \Delta_{t_{l+1}}\gamma^\pi,\Delta_{t_{i+j}}\gamma^\pi\rangle.
\end{equation}
Since the right-hand side of \cref{eq: explicit_scheme} only involves certain $K_{\gamma^\pi}(t_r,t_s)$ for which $r-s< j$, the system is explicitly solvable. Na\"ively solving this system incurs a time complexity of $\OO\big(n^3\big)$ in the length of the path. However, just as for the finite difference scheme proposed in \cite{SigPDE} for the ordinary signature kernel, rather than solving the grid in row and column order, we can update a whole anti-diagonal of the solution grid $\Pi$ at once since there is no shared dependency between points on the same anti-diagonal. It is possible to exploit this parallelisation using \textrm{GPU} architecture, provided that the length $n+1$ of $\gamma$ does not exceed the total number of threads on the \textrm{GPU}. The overall time complexity is then $\OO\big(dn^2\big)$ to populate $\Pi$, where $d$ is the dimension of the path. The following result guarantees that the preceding scheme converges to the true solution as the mesh of the partition $\pi$ shrinks to zero.
\begin{theorem}\label{thm: numerical_conv}
    Let $\gamma\in\XX$ and let $\pi=\{0=t_1<\dots<t_n<t_{n+1}=T\}$, $\gamma^\pi$ and $K_{\gamma^\pi}$ be as previously defined. Then
    \begin{equation}
        \left|K_{\gamma}(0,T) - K_{\gamma^\pi}(0, T)\right|\leq 16\norm{\gamma}_{1;[0,T]}\exp\big(4\norm{\gamma}_{1;[0,T]}\big)\max_{1\leq i \leq n}\norm{\gamma}_{1;[t_i,t_{i+1}]}.
    \end{equation}
\end{theorem}
\begin{proof}
    See \cref{sec: app_a}.
\end{proof}
\begin{remark}
    Note that for any path $\gamma\in\XX$, along any sequence of partitions $\pi_n$, with $|\pi_n|\to0$, it holds that $\lim_{n\to \infty}\max_{1\leq i\leq n}\norm{\gamma}_{1;[t_i,t_{i+1}]}=0$, so that \cref{thm: numerical_conv} guarantees convergence along any sequence of partitions with vanishing mesh size.
\end{remark}
\begin{remark}[Exact Solution for Straight Lines and $d=1$]\label{rem: exact_sol}
    We fix $T=1$ without loss of generality, since the signature is invariant under time reparameterisation. There are two situations in which the solution to \cref{eq: func_eq} is explicit. Firstly, assume that $\gamma$ is a straight line. That is, for $t\in [0,1]$
\begin{equation*}
    \gamma_t = tv, \quad\text{for some } v\in\R^d,
\end{equation*}
so that $\gamma_t^\prime\equiv v$ and $\langle \dif\gamma_s,\dif\gamma_t\rangle = \norm{v}^2\dif s \dif t$ for all $s,t\in [0, 1]$. In this case \cref{eq: func_eq} reduces to
\begin{equation*}
    K_\gamma(s,t)=1-\int_s^t\int_s^tK_\gamma(s, u)K_\gamma(u,r)\norm{v}^2\dif u \dif r,
\end{equation*}
which depends only on $\norm{v}$. Hence, the resulting solution will be the same as for the $\R$-valued path $\sigma_t = t\norm{v}$. Recalling the expansion \cref{eq: expansion}, we see that
\begin{equation*}
    K_\gamma(s,t)=K_\sigma(s,t)=\sum_{k=0}^\infty i^{2k} C_{2k} \frac{\norm{(t-s)v}^{2k}}{(2k)!}=\Ex\Big[\exp\big(iR\norm{(t-s)v}\big)\Big],
\end{equation*}
for a semi-circular random variable $R$ with radius $2$. The characteristic function of a semi-circular random variable is well known to be given by $\frac{J_1(2t)}{t}$, where $J_1(\cdot)$ is a Bessel function of the first kind. It is possible to exploit this explicit nature to fill in the part the solution grid $\Pi$ at any point $(t_i,t_{i+j})$ for which a path $\gamma$ is a straight line on the interval $[t_i,t_{i+j}]$. Finally, the solution is also explicit for any path in $d=1$, since the signature then depends only on the increment $\gamma_1-\gamma_0$.
\end{remark}
\subsection{Implicit Scheme}\label{sec: implicit}
The explicit scheme in \cref{sec: explicit} was based on left-point Riemann-Stieltjes integration. If one instead considers right-point Riemann-Stieltjes integration, one obtains the following implicit scheme. By expanding \cref{eq: func_eq} with right-point integrals, we obtain the following system of equations for $(t_i,t_{i+j})\in\Pi$.
\begin{equation*}
\begin{split}
   K_{\gamma^\pi}^r(t_i,t_{i+j})=1-\sum_{l=1}^{i+j}\sum_{k=1}^lK_{\gamma^\pi}^r(t_i,t_{i+k})K_{\gamma^\pi}^r(t_{i+k}t_{i+l})\langle \Delta_{t_{i+k}}\gamma^\pi,\Delta_{t_{i+l}}\gamma^\pi\rangle
\end{split}
\end{equation*}
with the boundary condition $K
_{\gamma^\pi}^r(t_i,t_i)=1$ for all $t_i\in\pi$. Just as for the explicit scheme, we can exploit the recursive nature of the equations to write
\begin{equation*}
K_{\gamma^\pi}^r(t_i,t_{i+j})=K_{\gamma^\pi}^r(t_{i},t_{i+j-1}) - \sum_{k=1}^jK_{\gamma^\pi}^r(t_i,t_{i+k})\langle \Delta_{t_{i+k}}\gamma^\pi,\Delta_{t_{i+j}}\gamma^\pi\rangle.
\end{equation*}
Noting that $K_{\gamma^\pi}^r(t_i,t_{i+j})$ appears once on the right-hand side, we can re-arrange to express $K_{\gamma^\pi}^r(t_i,t_{i+j})$ as
\begin{equation}
K_{\gamma^\pi}^r(t_i,t_{i+j})=\frac{K_{\gamma^\pi}^r(t_{i},t_{i+j-1}) - \sum\limits_{k=1}^{j-1}K_{\gamma^\pi}^r(t_i,t_{i+k})\langle \Delta_{t_{i+k}}\gamma^\pi,\Delta_{t_{i+j}}\gamma^\pi\rangle}{1+\langle \Delta_{t_{i+j}}\gamma^\pi,\Delta_{t_{i+j}}\gamma^\pi\rangle}
\end{equation}
\begin{remark}
    There are multiple variations on the implicit and explicit schemes to estimate the double integral in \cref{eq: func_eq}. For example, one could implement a Stratonovich style integral, where one averages the integrand at both left and right points. Alternatively, one could write $\langle \dif\gamma_u,\dif\gamma_r\rangle=\langle\gamma^\prime_u,\gamma^\prime_r\rangle \dif u \dif r$, so that the influence of the path $\gamma$ is in the integrand and not the integrator. In practice, these schemes all appear to have the same order of convergence, but with different constants.
\end{remark}
\subsection{Randomised CDEs}\label{sec: random}
We can also consider a randomised CDE approach to approximating $K_\gamma$, just as in \cite{SKlimit} for $K_{\text{sig}}$. Let $\hat{\gamma}^\pi$ be a piecewise linear approximation to $\gamma$ on $\pi$ and let $\big\{Z_{\hat{\gamma}^\pi}^{j,N}\big\}_{j=1}^M$ be independent solutions to \cref{eq: hermit}. Then we define the randomised approximation
\begin{equation}
    \text{r}K_{\hat{\gamma}^\pi}^{M, N}(s,t):=\frac{1}{M}\sum_{j=1}^M\frac{1}{N}\tr{Z_{\hat{\gamma}^\pi}^{j,N}(s,t)}.
\end{equation}
By \cref{thm: fpcf_eq} we have the convergence
\[
\lim_{N\to\infty}\lim_{M\to\infty}\text{r}K_{\hat{\gamma}^\pi}^{M, N}(s,t) = K_{\hat{\gamma}^\pi}(s,t).
\]
We choose the piecewise-linear approximation as then the solutions $Z_{\hat{\gamma}^\pi}^{j,N}$ are explicit and given as the product of matrix exponentials; see \cite[Section 2.2]{PCFGAN} for details.
\subsection{Comparing Numerical Solvers}
Given the stark contrast between the numerical approaches outlined in \cref{sec: explicit,sec: implicit} and the the Monte-Carlo approach of \cref{sec: random}, it is natural to compare their relative benefits and limitations. Of particular interest is the scaling behaviour of the computational time in the length of a path, its dimension, and the hyper-parameters of $\text{r}K^{M,N}$. All experiments were performed on an NVIDIA GeForce RTX 2060 with Max-Q Design. 
\begin{remark}
    This section is merely intended to compare the scaling of computational costs, and does claim superiority of one method over the other.
\end{remark}

In the following, the true underlying path $\gamma$ will be a piecewise-linear interpolation of a fractional Brownian motion with Hurst parameter $H=0.75$ on some uniform partition $\pi$ of $[0,1]$. For the implicit numerical scheme from \cref{sec: implicit}, we examine the convergence to $K_{\gamma}$ for $K_{\gamma^{\pi(\lambda)}}^r$ as $\lambda\to\infty$, where $\pi{(\lambda)}$ is a dyadic refinement of $\pi$ of order $\lambda \in \N$. For $\text{r}K_{\gamma}^{M,N}$, we fix $M=50$ and vary the dimension $N$ of the random matrices. All random matrices will be generated according to the Gaussian unitary ensemble (GUE).
\begin{figure}[H]
\centering
    \begin{subfigure}{.8\textwidth}
    \centering
    \includegraphics[width=1\linewidth]{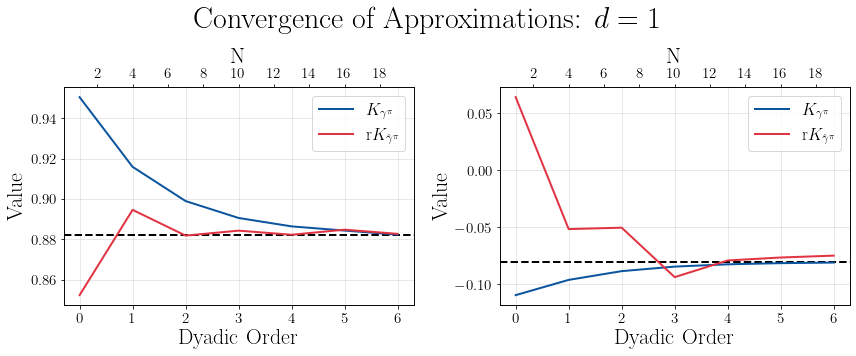}
    \caption{Convergence in $d=1$.}
    \label{fig: 1D_H_0.75}
    \end{subfigure}\\
\begin{subfigure}{.8\textwidth}
\centering
\includegraphics[width=1\linewidth]{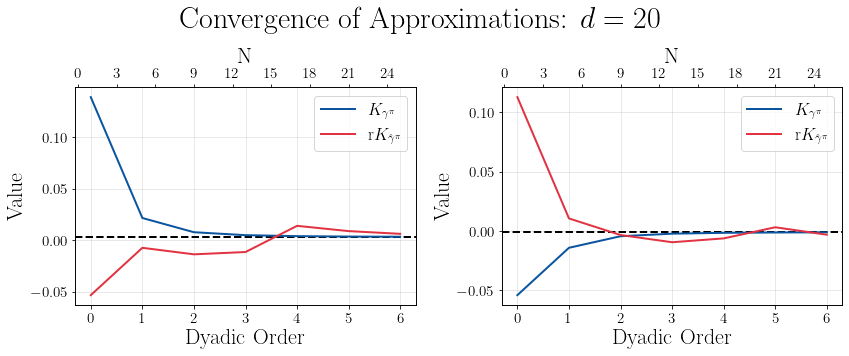}
\caption{Convergence in $d=20$.}
\label{fig: 20D_H_0.75}
\end{subfigure}
\caption{Convergence for the implicit numerical scheme $K_{\gamma^{\pi(\lambda)}}^r$ and the randomised scheme $\text{r}K_{\gamma}^{M,N}$ for $d=1$ and $d=20$. For each dimension, we generate two examples of a piecewise-linear interpolation of fractional Brownian motion with Hurst parameter $H=0.75$ on a uniform partition $\pi$ of $[0,1]$ of size $15$. The blue line corresponds to the lower $\mathrm{x}-$axis, indexed by the dyadic order, and the red line corresponds to the upper $\mathrm{x}-$axis, indexed by the dimension of the matrices $A_i^{M,N}$; the number of Monte-Carlo samples is fixed at $M=50$. The dotted black line is the true value $K_\gamma$, computed in $d=1$ by \cref{rem: exact_sol} and by $\text{r}K_{\gamma}^{125, 450}$ in $d=20$.}
\label{fig: convergence}
\end{figure}
From \cref{sec: implicit}, we know that the computational time for $K_{\gamma^{\pi}}^r$ scales as $\OO(n^2)$ in the size of $\pi$ if computed on a \textrm{GPU}. In contrast, the randomised approach will scale as $\OO(n)$. However, due to the increased computational costs involved (matrix exponentiation, matrix multiplication and Monte-Carlo sampling), it can be more costly to compute $\text{r}K_{\gamma}^{M, N}$. Furthmermore, both $K_{\gamma^{\pi}}^r$ and $\text{r}K_{\gamma}^{M, N}$ scale as $\OO(d)$ in the dimension of $\gamma$, although the constants involved are different. For $K_{\gamma^{\pi}}^r$, $n^2$ dot products of increments in $\R^d$ need to be computed. And for $\text{r}K_{\gamma}^{M, N}$, there are $nd$ scalar-matrix multiplications $\R\times\R^{N\times N}\to \R$.
\begin{figure}[H]
\centering
    \begin{subfigure}{.45\textwidth}
    \centering
    \includegraphics[width=1\linewidth]{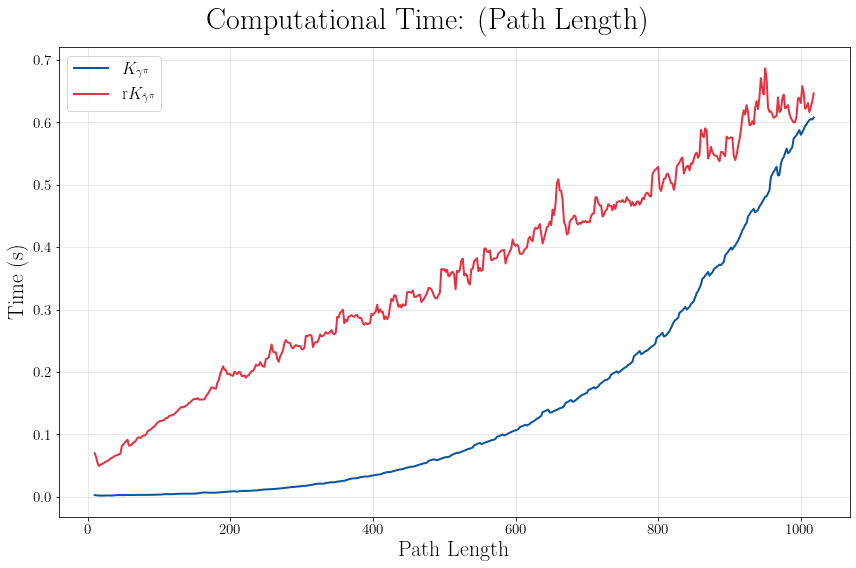}
    \caption{Time vs. path length in $d=10$.}
    \label{fig: compute_n}
    \end{subfigure}
\begin{subfigure}{.45\textwidth}
\centering
\includegraphics[width=1\linewidth]{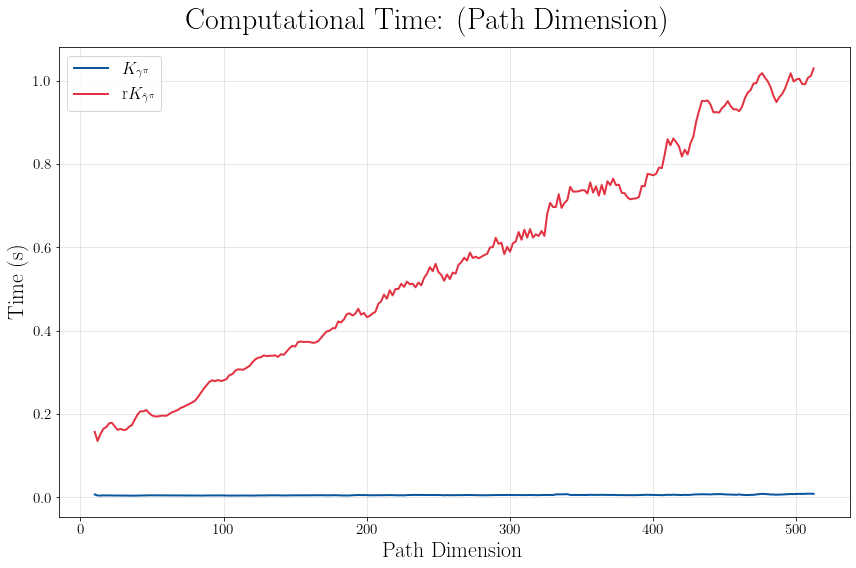}
\caption{Time vs. dimension with $128$ time steps.}
\label{fig: compute_d}
\end{subfigure}
\caption{A comparison of the computational time scaling for the implicit numerical scheme and the randomised scheme when path length and path dimension are varied. The hyper-parameters of $\text{r}K_{\gamma}^{M,N}$ are fixed at $N=20$ and $M=10$.}
\label{fig: compute}
\end{figure}

Finally, we note the impact of the choice of hyper-parameters $M$ and $N$ on the computational time for $\text{r}K_{\gamma}^{M,N}$. Indeed, the computational time is linear in the number of Monte-Carlo samples (without additional parallelisation), but more than quadratic in the dimension $N$, even when leveraging efficient algorithms that exploit \textrm{GPU} parallelisation.
\begin{figure}[H]
    \centering
    \includegraphics[width=.5\textwidth]{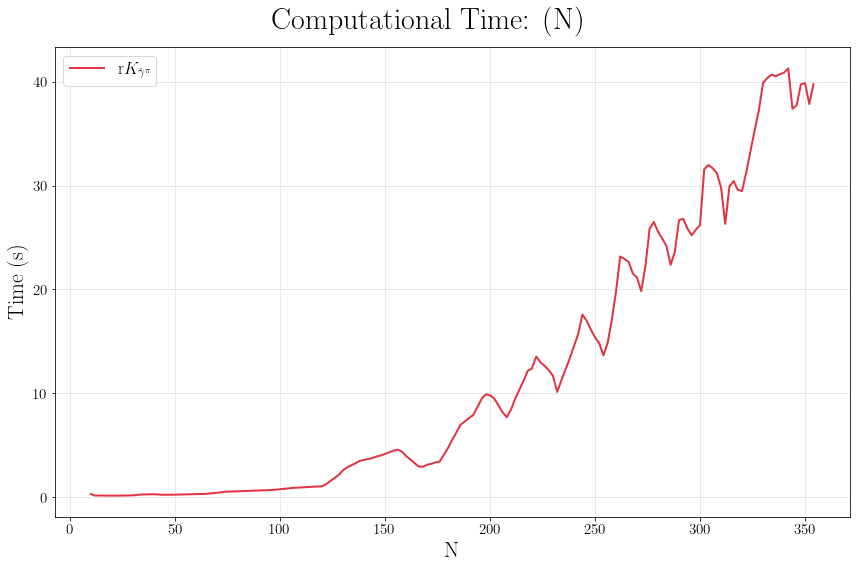}
    \caption{Computational time vs. N for a path with $128$ time steps in $d=10$. The remaining hyper-parameter of $\text{r}K_{\gamma}^{M,\cdot}$ is fixed at $M=2$.}
    \label{fig: compute_N}
\end{figure}
\section*{Funding}
TC has been supported by the EPSRC Programme Grant EP/S026347/1 and acknowledges the support of the Erik Ellentuck Fellowship at the Institute for Advanced Study. WFT has been supported by the EPSRC Centre for Doctoral Training in Mathematics of Random Systems: Analysis, Modelling and Simulation (EP/S023925/1). Both authors acknowledge the support of the Research Institute for Mathematical Sciences, an International Joint Usage/Research Center located in Kyoto  University during a visit in September 2023, when part of this work was conducted.
\clearpage
\printbibliography[heading=bibintoc,title={References}]
\clearpage
\appendix
\section{Proof of Convergence of Numerical Scheme}\label{sec: app_a}
In this section we prove \cref{thm: numerical_conv}. We follow the method of \cite[Appendix A.3]{KO}, wherein the authors prove convergence of numerical schemes for the ordinary signature kernel based on piecewise constant approximations. Throughout, let $\gamma\in\XX$ and let $\pi=\{0=t_1<\dots<t_n<t_{n+1}=T\}$ be a partition of $[0,T]$, with $\gamma^\pi$ defined by
\[
\gamma_t^\pi=\sum_{i=1}^n\id_{[t_i,t_{i+1})}\gamma_{t_i}.
\]
We interpret all integrals with respect to $\gamma^\pi$ as left-point Riemann-Stieltjes integrals.
\begin{lemma}
The signature of $\gamma^\pi$ is well defined and given by
\begin{equation}\label{eq: ISS}
\SS_{0,T}\big(\gamma^\pi\big)^m = \sum_{1\leq i_1<\dots< i_m\leq n}\big(\Delta_{t_{i_1+1}}\gamma^\pi\big)\otimes\dots\otimes\big(\Delta_{t_{i_m}+1}\gamma^\pi\big).
\end{equation}
\end{lemma}
\begin{remark}
    The object on the right-hand side of \cref{eq: ISS} is often referred to as the iterated-sums or discrete-time signature of $\gamma^\pi$. For details on the algebraic and analytic properties of the iterated-sums signature, we refer the reader to \cite{ISS,KO}.
\end{remark}
The following bound on the difference between the $m$\textsuperscript{th} term of the signature of $\gamma$ and $\gamma^\pi$ may be deduced from a careful reading of \cite[Lemma A.4 and Proposition A.5]{KO}.
\begin{lemma}\label{lem: KO_bound}
    For all $m\geq 2$, it holds that
    \begin{equation}
        \norm{\SS_{0,T}\big(\gamma\big)^m-\SS_{0,T}\big(\gamma^\pi\big)^m}\leq \frac{\norm{\gamma}_{1;[0,T]}^{m-1}}{(m-2)!}\max_{1\leq i \leq n}\norm{\gamma}_{1;[t_i,t_{i+1}]}.
    \end{equation}
\end{lemma}
\begin{proof}[Sketch of Proof]
    By \cite[Lemma A.4]{KO}, the difference $\SS_{0,T}\big(\gamma\big)^m-\SS_{0,T}\big(\gamma^\pi\big)^m$ is the $m$-\textsuperscript{th} iterated integral of $\gamma$ over the subset of the $m$-simplex where the projections of at least two coordinates lie in the same interval $[t_i,t_{i+1}]\in \pi$. By dividing this subset into disjoint regions, the authors of \cite{KO} obtain a bound on the difference in terms of the $1$-variation norm of $\gamma$ over each sub-interval $[t_i,t_{i+1}]$. By writing $\norm{\gamma}_{1;[0,T]}=\sum_{i=1}^n\norm{\gamma}_{1;[t_i,t_{i+1}]}$, it is then shown in \cite[Proposition A.5]{KO} that the this bound on the difference is smaller than $\norm{\gamma}_{1;[0,t]}\max_{1\leq i \leq n}\norm{\gamma}_{1;[t_i,t_{i+1}]}$ multiplied by the $(m-2)$\textsuperscript{th} term in the Taylor expansion of $\exp\big(\norm{\gamma}_{1;[0,T]}\big)$.
    \end{proof}
\begin{corollary}
    For any $\gamma\in\XX$ and $\pi$, we have the bound
    \[
    \left|K_{\gamma}(0,T) - K_{\gamma^\pi}(0, T)\right|\leq 16\norm{\gamma}_{1;[0,T]}\exp\big(4\norm{\gamma}_{1;[0,T]}\big)\max_{1\leq i \leq n}\norm{\gamma}_{1;[t_i,t_{i+1}]}.
    \]
\end{corollary}
\begin{proof}
    By the triangle inequality and \cref{lem: KO_bound}, we see that
    \begin{align*}
        \left|K_{\gamma}(0,T) - K_{\gamma^\pi}(0, T)\right|&= \left|\sum_{|\bm{I}|=0}^\infty (-1)^{\frac{|\bm{I}|}{2}}\varphi(\bm{I})(\gamma^\pi)\big(\SS_{0,T}^{\bm{I}}(\gamma)-\SS_{0,T}^{\bm{I}}(\gamma^{\pi})\big)\right|\\
        &\leq \sum_{|\bm{I}|=2}^\infty\varphi(\bm{I})\left|\SS_{0,T}^{\bm{I}}(\gamma)-\SS_{0,T}^{\bm{I}}(\gamma^{\pi})\right|\\
        &\leq  \sum_{m=2}^\infty C_m\norm{\SS_{0,T}\big(\gamma\big)^m-\SS_{0,T}\big(\gamma^\pi\big)^m}\\
        &\leq \sum_{m=2}^\infty 4^m\frac{\norm{\gamma}_{1;[0,T]}^{m-1}}{(m-2)!}\max_{1\leq i \leq n}\norm{\gamma}_{1;[t_i,t_{i+1}]}\\
        &\leq 16\norm{\gamma}_{1;[0,T]}\max_{1\leq i \leq n}\norm{\gamma}_{1;[t_i,t_{i+1}]}\sum_{m=0}^\infty \frac{\big(4\norm{\gamma}_{1;[0,T]}\big)^m}{m!}\\
        &=16\norm{\gamma}_{1;[0,T]}\exp\big(4\norm{\gamma}_{1;[0,T]}\big)\max_{1\leq i \leq n}\norm{\gamma}_{1;[t_i,t_{i+1}]}
    \end{align*}
\end{proof}

\end{document}